\documentclass[12pt,draftcls,onecolumn]{IEEEtran}

\usepackage{times,verbatim} 
\usepackage{amsmath, amssymb}
\usepackage{amsthm}
\usepackage{thmtools}
\usepackage{graphicx}
\usepackage{subfigure}
\usepackage{epstopdf}
\usepackage{algorithm}
\usepackage{algorithmic}
\usepackage{enumerate}
\usepackage{color}

\usepackage{standalone}
\usepackage{tikz}
\usetikzlibrary{shapes,shadows,calc,snakes}
\usetikzlibrary{decorations.pathmorphing,patterns}
\usetikzlibrary{decorations}
\usetikzlibrary{decorations.mypathmorphing}
\usepackage{pgfplots}

\newtheorem{theorem}{Theorem}[section]

\newtheorem{problem}[theorem]{Problem}
\newtheorem{proposition}[theorem]{Proposition}

\newtheorem{definition}[theorem]{Definition}
\newtheorem{example}[theorem]{Example}

\newtheorem{remark}[theorem]{Remark}
\numberwithin{equation}{section}

\renewcommand{\S}{{\mathbb{S}}}

\newcommand{\R}{{\mathbb{R}}}

\newcommand{\N}{{\mathbb{N}}}

\renewcommand{\O}{\mathcal{O}}
\newcommand{\M}{\mathcal{M}}

\renewcommand{\tilde}[1]{\widetilde{#1}}

\newcommand{\xhat}{\hat{x}}
\newcommand{\ahat}{\hat{E}}
\newcommand{\zhat}{\hat{z}}

\newcommand{\half}{\frac{1}{2}}
\newcommand{\norm}[1]{\left\Vert#1\right\Vert_2}
\renewcommand{\matrix}[1]{\begin{bmatrix}#1\end{bmatrix}}

\title{\LARGE \bf 
Event-Triggered State Observers for \\Sparse Sensor Noise/Attacks
}

\author{Yasser Shoukry and Paulo Tabuada
\thanks{Y. Shoukry and P. Tabuada are with the UCLA Electrical
Engineering Department, {\tt\small yshoukry@ucla.edu, tabuada@ee.ucla.edu}}%
\thanks{This work was partially sponsored by the NSF award 1136174 and
by DARPA under agreement number FA8750-12-2-0247. The U.S. Government
is authorized to reproduce and distribute reprints for Governmental purposes
notwithstanding any copyright notation thereon. The views and conclusions
contained herein are those of the authors and should not be interpreted
as necessarily representing the official policies or endorsements, either
expressed or implied, of NSF, DARPA or the U.S. Government.}
}

\begin{document}

\maketitle

\begin{abstract}
This paper describes two algorithms for state reconstruction from sensor 
measurements that are corrupted with sparse, but otherwise arbitrary, ``noise''.
These results are motivated by the need to secure cyber-physical systems against a
 malicious adversary that can arbitrarily corrupt sensor measurements.
The first algorithm reconstructs the state from a batch of sensor measurements
while the second algorithm is able to incorporate new measurements as they 
become available, in the spirit of a Luenberger observer. A distinguishing point
of these algorithms is the use of event-triggered techniques to improve the computational performance of the proposed algorithms.
\end{abstract}

\section{Introduction}

The security of Cyber-Physical Systems (CPSs) has recently become a topic of scientific inquiry in no small part due to the discovery of the Stuxnet malware, the most famous example of an attack on process control systems~\cite{stuxnet}.  Although one might be tempted to associate CPS security with large-scale and critical infrastructure, such as the power-grid and water distribution systems, previous work by the authors and co-workers has shown that even smaller systems, such as cars, can be attacked. It was shown in~\cite{YasserABS} how to attack the velocity measurements of anti-lock braking systems so as to force drivers to loose control of their vehicles.

In this paper we propose two state observers for discrete-time linear systems in
the presence of sparsely corrupted measurements. Sparse ``noise'' is
a natural model to describe the effect of a malicious attacker that has the
ability to alter the measurements of a subset of sensors in a feedback control
loop. While measurements originating from un-attacked sensors are ``noise'' free,
measurements from attacked sensors can be arbitrary: we make no assumption
regarding its magnitude, statistical description, or temporal evolution. Hence, the noise vector is sparse; its elements are either zero or arbitrary real numbers.

Several results on state reconstruction under sensor attacks have recently 
appeared in the literature. We classify the existing work in
two classes based on how the physical plant is modeled: 1)
steady-state operation (no dynamics) and 2) linear time-invariant dynamics.
In both classes the attacker is assumed to corrupt a few sensor measurements and thus its actions are adequately modeled as sparse noise.

The results reported in \cite{KalleGrid,LiuPowerAttack,SandbergPowerAttack,KosutPowerAttack,KimPowerAttack} fall in the first class -- steady state operation -- and address security problems in the context of smart power grid systems. Due to the steady state assumption, all of these results fail to exploit the constraints imposed by the continuous dynamics as done in this paper.

Representative work in the second class -- linear time-invariant dynamics -- includes~\cite{Bullo_TAC,Hamzaarxiv,HamzaAllerton,HamzaCDC}. The work reported in~\cite{Bullo_TAC} addresses the detection of attacks through monitors inspired by the fault detection and isolation literature. Such methods are better suited for small systems since the number of monitors grows combinatorially with the number of sensors. The work reported in~\cite{Hamzaarxiv,HamzaAllerton,HamzaCDC} draws inspiration from error correction over the reals \cite{DecodingOverReals} and compressive sensing \cite{CompressiveSensingBook} and formulates the secure state reconstruction problem as a $L_r \backslash L_1, r>1$ optimization problem.

The problem of reconstructing the state under sensor attacks is closely related to fault-tolerant state reconstruction. The robust Kalman filter, described in~\cite{BoydKF}, is the approach to fault-tolerant state reconstruction closer to the results in this paper, at the technical level. In robust Kalman filtering the state estimate updates are obtained by solving a convex $L_1$ optimization problem that is robust to outliers. With the advances in the computational power of current processors, the robust Kalman filter can be efficiently computed in real-time. However, no theoretical guarantees are known regarding the performance of this filter in the presence of malicious attacks.

In this paper, we extend the work in
\cite{Hamzaarxiv,HamzaAllerton,HamzaCDC} by focusing on \emph{efficient} algorithms in the sense of being implementable on computationally limited platforms. Rather than relying on classical algorithms for $L_r \backslash L_1, r>1$ optimization, as was done in~\cite{Hamzaarxiv,HamzaAllerton,HamzaCDC}, we develop in this paper customized gradient-descent algorithms which have lightweight implementations. The computational efficiency claims are supported by numerical
simulations showing an order of magnitude decrease in the computation time. 

The proposed algorithms  reconstruct both the state as well as the sparse
noise/attack signal. Hence, they can  be seen as an extension of compressive
sensing techniques to the case where part of the signal to be reconstructed is
sparse and the other part is governed by linear dynamics. A similar problem is
studied in \cite{sparseStreaming} where the recovery of a sparse streaming signal with linear dynamics is discussed. Although the work in~\cite{sparseStreaming} also exploits the linear dynamics, it is not applicable to the state reconstruction problem where the sparse signal models an attack for which no dynamics is known.

Technically, we make the following contributions:
\begin{itemize}
\item The reconstruction or decoding of compressively sensed signals is
characterized by properties such as the \emph{restricted isometry} or the
\emph{restricted eigenvalues} \cite{CompressiveSensingBook}. We show that the relevant notion in our case is a strong notion of observability.
\item We extend one of the algorithms previously proposed for the reconstruction or decoding 
of compressively sensed signals \cite{AcceleratedIHT} to the case where part of the signal is sparse while the rest is governed by linear dynamics.
\item We propose a recursive implementation of the method discussed in the previous bullet so
that new measurement information can be used as it becomes available, in the spirit of a Luenberger observer.
\end{itemize}

The rest of this paper is organized as follows. Section \ref{sec:problem}
formally introduces  the problem under consideration. The notions of
$s$-observability and $s$-restricted eigenvalues are introduced in Section
\ref{sec:observ}. The main results of this paper which are the Event-Triggered
Projected Gradient-Descent algorithm and the Event-Triggered Projected 
Luenberger Observer, along with their convergence properties, are presented in Sections \ref{sec:tgd} and \ref{sec:luenberger}, respectively. 
Simulation results for the proposed algorithms are shown in
Section \ref{sec:results}. Finally, Section \ref{sec:conclusion} concludes this
paper.

\section{The Secure State Reconstruction Problem}
\label{sec:problem}

\subsection{Notation}
The symbols $\N_0, \R,$ and $\R^+$ denote the set of natural, real, and positive real numbers, respectively. Given two vectors $x \in \R^{n_1}$ and $y\in \R^{n_2}$, we denote by $(x,y) \in \R^{n_1 + n_2}$ the vector $\matrix{x^T & y^T}^T$. We also use the notation $z_x$ and $z_y$ to denote the natural projection of the vector $z = (x,y)$ on its first and second component, respectively.

If $S$ is a set, we denote by $|S|$ the cardinality of $S$. The support of a vector $x\in \R^n$, 
denoted by $\textrm{supp}(x)$, is the set of indices of the non-zero elements of $x$.
We call a vector $x \in \R^n$ $s$-sparse, if  $x$ has at most $s$ nonzero
elements, i.e., if  $\vert \textrm{supp}(x)\vert \le s$. A vector $z = (x_1, x_2, \hdots, x_p) \in \R^{np}$ 
is called cyclic $s$-sparse, if each
$x_i \in \R^{n}$ is $s$-sparse for all $i \in \{1,2, \hdots ,p \}$ and
$\textrm{supp}(x_1) = \textrm{supp}(x_2) = \hdots = \textrm{supp}(x_p)$. With some abuse of notation,
we use $s$-sparse to denote cyclic $s$-sparse.

For a vector $x \in \R^n$, we denote by $\norm{x}$ the $2$-norm of $x$ and by $\norm{M}$ the induced $2$-norm of a matrix $M \in \R^{m \times n}$. We also denote by $M_i \in \R^{1\times n}$ the $i${th}
row of $M$. For a set $\Gamma \subseteq \{1, \hdots, m\}$, we denote by
$M_{\Gamma} \in \R^{\vert \Gamma\vert \times n}$ the matrix obtained from $M$ by  removing all the rows except those
indexed by $\Gamma$. We also denote by $M_{\overline{\Gamma}} \in
\R^{(m-|\Gamma|) \times n}$ the matrix obtained from $M$ by removing the rows indexed by
the set $\Gamma$, for example, if $m = 4$, and $\Gamma = \{1,2\}$, then:
$$M_\Gamma =\matrix{M_1 \\ M_2} \mbox{ and } M_{\overline{\Gamma}} =
\matrix{M_3\\M_4}.$$

For any finite set $S=\{s_1,s_2,\hdots,s_k\}\subset \N$ we denote by $rS$ and $r+S$, $r\in \N$, the sets $rS=\{rs_1,rs_2,\hdots,rs_k\}$ and $r + S = \{r + s_1, r+s_2, \hdots r + s_k \}$, respectively. Finally we denote the set of eigenvalues, the minimum eigenvalue, and the
maximum eigenvalue of a symmetric matrix $M$ by $\lambda{\{M\}}$, $\lambda_{\min}\{M\}$,
and $\lambda_{\max}\{M\}$ respectively.

\subsection{Dynamics and Attack Model}
Consider the following linear discrete-time control system where
$x(t)\in \R^n$ is the system state at time $t \in \N_0$, $u(t) \in
\R^m$ is the system input, and $y(t) \in \R^p$ is the observed measurement:
\begin{align}
	x(t+1) &= A x(t) + B u(t), \label{eq:sys_state} \\
	y(t) &= C x(t) + a(t).	\label{eq:sys_out}	
\end{align}
The matrices $A, B$, and $C$ have appropriate dimensions and
$a(t) \in \R^p$ is a $s$-sparse vector modeling how an attacker changed the sensor 
measurements at time $t$. If sensor $i \in \{1, \hdots ,p\}$ is attacked then
the $i${th} element in the vector $a(t)$ is non-zero otherwise the $i${th}
sensor is not attacked. Hence, $s$ describes the number of attacked sensors. 
Note that we make no assumptions on the vector $a(t)$ other than being $s$-sparse. 
In particular, we do not assume bounds, statistical properties, nor restrictions on 
the time evolution of the elements in $a(t)$. The value of $s$ is also not
assumed to be known although we assume the knowledge of an upper bound.
The set of sensors the attacker has access to is assumed to remain constant over time (and has cardinality at most $s$). However, the attacker has complete freedom in deciding which sensor or sensors in this set are attacked and when, including the possibility of attacking all of them at all times.

Our objective is to simultaneously construct a delayed version of the state,
$x(t-\tau+1)$, and the \emph{cyclic $s$-sparse} attack vector $E(t) = (a(t-\tau+1), a(t-\tau+2),
\hdots, a(t))$ from the measurements $y(t-\tau+1), y(t-\tau+2), \hdots, y(t)$.

It is worth explaining the cyclic sparse nature of $E(t)$ in more detail.
Consider the attack vector $E(t) = (a(t-\tau+1),
a(t-\tau+2), \hdots, a(t))$ and reformat this data as the matrix $\tilde{E}(t) \in \R^{p\times \tau}$ where column $i$ is given by $a(t-i+1)$.
Since we assume that the set of sensors under attack does not change over
time, the sparsity pattern appears in $\tilde{E}(t)$. The rows corresponding to
the un-attacked sensors will only have zeros, while the rows
corresponding to the attacked sensors will have arbitrary (zero or non-zero) elements.
\begin{example}
Consider a system with $4$ sensors, $\tau=4$,
and an attack on the second and third sensors. The attack
matrix  $\tilde{E}(t)$ will be of the form:
\begin{align*}
	\tilde{E}(t) = \matrix{0 & 0 & 0 & 0 \\ 2 & 5 & 6 & 10 \\ 4 & 8 & 0 & 12 \\ 0 & 0 & 0 & 0}.
\end{align*}
The cyclic sparsity structure appears once the attack matrix $\tilde{E}(t)$ is
reshaped as the vector $E(t)$:
\begin{align*}
	E(t) = \matrix{0 & 2 & 4 & 0 & 0 & 5 & 8  & 0 & \hdots }^T.
\end{align*}
\label{ex:example1}
\end{example}
Since cyclic $s$-sparse vectors pervade this paper, it is convenient to denote them by a special symbol.

\begin{definition}[\textbf{Cyclic $s$-sparse set $\S_{s}$}]
The subset of $\R^{p\tau}$ consisting of the vectors that are cyclic $r$-sparse for all $r\in \{0,1,\hdots,s\}$, $s\le p$, is denoted by $\S_{s}$.
\end{definition}

Using this cyclic sparsity notion, we pose two problems which will lead to
the two proposed algorithms.

\subsection{Static Batch Optimization}
By collecting $\tau\in \N$ observations with $\tau \le n$, we can write the output equation as:
\begin{align*}
	\tilde{Y}(t)  &= \O x(t-\tau+1) + E(t) + F U(t) \\
		&= \matrix{\O & I} \matrix{x(t-\tau+1) \\ E(t)} + F U(t)\\
		&= Q z(t) + F U(t),
\end{align*}
where:
\begin{align}	
z(t) &= \matrix{x(t-\tau+1) \\ a(t-\tau+1) \\ a(t-\tau+2) \\ \vdots \\ a(t)}
	= \matrix{x(t-\tau+1) \\ E(t)},
	E(t) = \matrix{ a(t-\tau+1) \\ a(t-\tau+2) \\ \vdots \\ a(t) }, \nonumber\\
	\O &= \matrix{ C \\ CA \\ \vdots \\ CA^{\tau-1} },	Q = \matrix{ \O & I }, \label{eq:defQ}\\
	F &= \matrix{ 0 & 0 & \hdots & 0 & 0\\ CB & 0 & \hdots & 0 & 0 \\ \vdots &
	& \ddots &  & 
	\vdots
	\\
	CA^{\tau-2}B & CA^{\tau-3}B & \hdots & CB & 0 } 
	,\tilde{Y}(t) = \matrix{ y(t-\tau+1) \\ y(t-\tau+2) \\ \vdots \\ y(t) },
	U(t) = \matrix{ u(t-\tau+1) \\ u(t-\tau+2) \\ \vdots \\ u(t) }. \nonumber
\end{align} 
Since all the inputs in the vector $U(t)$ are known, we can further simplify
the output equation to:
$$Y(t)=Qz(t),$$
where $Y(t)=\tilde{Y}(t)-FU(t)$.

\begin{problem}{\textbf{(Static Batch Optimization)}}
For the linear control system defined by~\eqref{eq:sys_state}
and \eqref{eq:sys_out} construct the estimate $\hat{z}=(\hat{x},\hat{E})$, where
$\hat{x}\in \R^n$ is the state estimate and $\hat{E}\in \S_s$ is the attack
vector estimate, obtained as the solution of the following optimization problem:
\begin{align*}
\arg\min_{\hat{z}\in \R^n\times \S_s} \half \left\Vert
Y-Q\hat{z}\right\Vert_2^2.
\end{align*}
\label{prob:opt}
\end{problem}
We dropped the time $t$ argument since 
this optimization problem is to be solved at every time instance. We note that we seek a solution in the non-convex set $\R^n\times\mathbb{S}_s$ and no closed-form solution is known for this problem. Note also that this optimization problem asks for the reconstruction of a delayed version of the state $x(t-\tau+1)$. However, we can always reconstruct of the current state $x(t)$ from $x(t-\tau+1)$ by recursively rolling the dynamics forward in time. Alternatively, when $A$ is invertible, we can directly recover $x(t)$ by re-writing the measurement equation as a function of $x(t)$.

As a final remark we note that it follows from the Cayley-Hamilton theorem that there is no loss of generality in taking the number of collected measurements $\tau$ to be no greater than the number of states $n$.

\subsection{Luenberger-like Observer}

Problem \ref{prob:opt} asks for the reconstruction of $z(t)$ using a batch approach based on the
measured data in the vector $Y$. On computationally restricted platforms we may be faced with the
difficulty of having to process new measurements before being able to compute a
solution to Problem \ref{prob:opt}. It would then be preferable to reconstruct
$z(t)$ using an algorithm that can incorporate new measurements as they become
available. This motivates the following problem.

\begin{problem}{\textbf{(Luenberger-like Observer)}}
For the linear control system defined by~\eqref{eq:sys_state} and \eqref{eq:sys_out} construct a dynamical system:
$$\zhat(t+1) = f(\zhat(t),U(t),Y(t)),$$
such that:
$$\lim_{t\to\infty} ( z^*(t) - \zhat(t) ) = 0,$$
where $z^*(t)=(x^*(t-\tau+1),E^*(t))\in \R^n\times \S_s$, $x^*(t)$ is the solution of~(\ref{eq:sys_state}) under the inputs $U(t)$, and $Y(t)$ is the sequence of the
last $\tau$ observed outputs corrupted by $E^*(t)$.
\label{prob:observer}
\end{problem}


\section{$s$-Sparse Observability and The Restricted Eigenvalue}
\label{sec:observ}
Recall that Problem~\ref{prob:opt} asks for the minimizer of $\half\Vert
Y-Q\hat{z}\Vert_2^2$. Since the matrix $Q$  has a non trivial kernel, there exist
many pairs $\zhat=(\xhat, \ahat)$ which solve this problem. In this section we look closely
at the fundamental question of uniqueness of solutions. 

\subsection{$s$-Sparse Observability}
We start by introducing the following notion of observability.
\begin{definition}{\textbf{($s$-Sparse Observable System)}}
The linear control system defined by~\eqref{eq:sys_state}
and \eqref{eq:sys_out}
is said to be $s$-sparse observable
if for every set $\Gamma\subseteq\{1,\hdots,p\}$ with $|\Gamma| = s$, the pair
$(A,C_{\overline{\Gamma}})$ is observable.
\end{definition}

In other words, a system is $s$-sparse observable if it remains 
observable after eliminating any choice of $s$ sensors. This strong
notion of observability underlies most of the results in this paper and can also be expressed using the notion of strong 
observability for linear systems described in \cite{KratzStrongObsv,strongObservability}. To make this connection explicit, consider a linear system $(A;B = 0;C;D = I_a)$ where the attack signal $a(t)$ is regarded as an input and $I_a$ is the
diagonal matrix defined so that the $i${th} element on the diagonal is zero whenever $a_i(t)$ is zero and is one otherwise. Since in our formulation we assume no knowledge of the support of the attack signal $a(t)$, we need to consider all the different supports for the attack vector. It is then not difficult to see that a linear system is $s$-sparse observable if and only if the systems $(A;B = 0;C;D = I_a)$ are strong observable for all the matrices $I_a$ obtained by considering attack vectors with all the possible supports.

We use the notion of $s$-sparse observability to characterize the uniqueness of solutions to Problem \ref{prob:opt}.

\begin{theorem}{\textbf{(Existence and uniqueness of solutions to Problem \ref{prob:opt})}} Problem \ref{prob:opt} has a unique solution, i.e., the function $\half\Vert
Y-Q\hat{z}\Vert_2^2$ has a unique minimum on the set $\R^n\times \S_s$, if and only
if the linear dynamical system defined by \eqref{eq:sys_state} and
\eqref{eq:sys_out} is $2s$-sparse observable.
\label{th:injectivity}	
\end{theorem}

\begin{proof}
We first note that $\norm{Y-Q z}^2$ is always non-negative
and it becomes zero whenever $x$ is the true state and $E$ is the
true attack vector.  Hence, $\Vert
Y-\mathcal{O}x-E\Vert_2^2$ has a unique minimum if and only if the equality $Y=Qz = \mathcal{O}x+E$ 
only holds for the true state and the true
attack vector. However, this is equivalent to injectivity of the map
$f:\R^n\times \S_s\to \R^{p \tau}$ defined by $f(x,E)=\mathcal{O}x+E$.

To prove the stated result we assume, for the sake of contradiction, that $f$ is not
injective and $(A,C)$ is $2s$-sparse observable. Since $f$ is not injective there exist $(x,E), (x',E')\in \R^n\times \S_s$ such that $f(x,E)=f(x',E')$. Moreover, it follows from the definition of $f$ that $x\ne x'$. We now note that:
\begin{align*}
f(x,E)  =  f(x',E')
\Leftrightarrow  \mathcal{O}x+E=\mathcal{O}x'+E'
\Leftrightarrow  \mathcal{O}(x-x')=E'-E.
\end{align*}
Since the support of $a$ and $a'$ is at most $s$, the support of $E'-E$ is at
most $2s$. Let \mbox{$\Gamma = \textrm{supp}(E'-E)$} be the set consisting of the indices where
$E'-E$ is supported. Then, $\mathcal{O}(x-x')=E'-E$ implies $\mathcal{O}_{\overline{\Gamma}}(x-x')=0$ which in
turn implies that the pair $(A,C_{\overline{\Gamma}})$ is not observable since $x\ne x'$, a
contradiction.

Conversely, if the pair $(A,C_{\overline{\Gamma}})$ is not observable there is a non-zero
vector $v\in \ker\mathcal{O}_{\overline{\Gamma}}$. Fix $x\in \R^n$ and let $x'=x+v$. Since
$\mathcal{O}_{\overline{\Gamma}}(x-x')=0$, the support of $\mathcal{O}(x-x')$ is at
most $\vert\Gamma\vert=2s$. Split $\Gamma$ in two sets $\Gamma_1$ and $\Gamma_2$
so that $\Gamma=\Gamma_1\cup\Gamma_2$, $\vert\Gamma_1\vert\le s$ and
$\vert\Gamma_2\vert\le s$. We can now define $E$ by making its $i$th entry
$E_i$ to be equal to $(\mathcal{O}(x-x'))_i$ if $i\in \Gamma_1$ and zero otherwise.
Similarly, we define $E'$ by making its $i$th entry $E'_i$ to be equal to
$(\mathcal{O}(x-x'))_i$ if $i\in \Gamma_2$ and zero otherwise. As defined, the
support of $E$ and $E'$ is at most $s$ and the equality $f(x,E)=f(x',E')$ holds
thereby showing that $f$ is not injective.
\end{proof}

Although the notion of $s$-sparse observability is of combinatorial nature,
since we have to check observability of all possible pairs
$(A,C_{\overline{\Gamma}})$, it does clearly illustrate a fundamental
limitation: it is impossible to correctly reconstruct the state whenever $p/2$ or more sensors are attacked. Indeed, suppose that we have
an even number of sensors $p$ and $s = p/2$ sensors are attacked. Theorem
\ref{th:injectivity} requires  the system to still be observable after removing
$2s = p$ rows from the matrix $C$ which leads to $C_{\overline{\Gamma}}$ being
the linear transformation  mapping every state to zero.  This fundamental
limitation is consistent with and had been previously reported in \cite{Hamzaarxiv,HamzaAllerton,strongObservability}.

In \cite{Hamzaarxiv,HamzaAllerton}, the possibility of reconstructing the state under sensor attacks is characterized by Proposition 2 and under certain assumptions on the $A$ matrix by Proposition 4. The characterization based on $s$-sparse observability complements the characterizations in \cite{Hamzaarxiv,HamzaAllerton} in the following sense. Proposition 2 in \cite{Hamzaarxiv,HamzaAllerton} requires a test to be performed for every state $x\in \R^n$ and does not lead to an effective algorithm. In contrast, $s$-sparse observability only requires a finite, albeit large, number of computations. Proposition 4 in \cite{Hamzaarxiv,HamzaAllerton} requires an even smaller number of tests than $s$-sparse observability but only applies under additional assumptions on the $A$ matrix. Moreover, $s$-sparse observability connects the state reconstruction problem under sensor attacks to the well known systems theoretic notion of observability.
\subsection{$s$-Restricted Eigenvalue}
\label{sec:restEval}

While $s$-sparse observability provides a qualitative characterization of the existence
and uniqueness of solutions to Problem \ref{prob:opt}, in this section we discuss a more
quantitative version termed the \emph{$s$-restricted eigenvalue property} \cite{restrictedEvalue}. This property is 
directly related to the possibility of solving Problem \ref{prob:opt} and
Problem \ref{prob:observer} using gradient descent inspired methods.

\begin{definition}{\textbf{($s$-Restricted Eigenvalue of a Linear Control System)}}
For a given set \mbox{$\tilde{\Gamma}_s \subseteq \{1, \hdots, p\}$} with $\vert \tilde{\Gamma}_s\vert = p - s$, let $\Gamma_s$ be defined by:
\begin{align}
\Gamma_s=\tilde{\Gamma}_s\cup p\tilde{\Gamma}_s\cup 2p\tilde{\Gamma}_s\cup\hdots\cup (\tau -1)p\tilde{\Gamma}_s
\label{eq:gamma_def}
\end{align}
and define (with some abuse of notation) the matrix $Q_{\overline{\Gamma}_s} \in \R^{p \tau \times (n + s\tau)}$ as the matrix obtained from $Q = \matrix{\O & I}$ by removing from $I$ the columns indexed by $\Gamma_s$, i.e.:
\begin{align}
	Q_{\overline{\Gamma}_s} = \matrix{\O & (I_{\overline{\Gamma}_s})^T }.
\label{eq:Qz_gamma}	
\end{align}
The $s$-restricted eigenvalue of the control system defined by~(\ref{eq:sys_state}) and~(\ref{eq:sys_out})
is the the smallest eigenvalue of all the matrices $Q_{\overline{\Gamma}_s}^TQ_{\overline{\Gamma}_s}$ obtained by considering all the different sets $\tilde{\Gamma}_s$.
\end{definition}

The $s$-restricted eigenvalue of a control system can be related to the
$s$-observability as follows.

\begin{proposition}{\textbf{(Non-zero Restricted Eigenvalue)}}
Let the linear control system, defined by~\eqref{eq:sys_state}
and~(\ref{eq:sys_out}), be $2s$-sparse observable. There exists a $\delta_{2s} \in \R^+$ such
that for every $z = (x,E) \in \R^n \times \S_{2s}$ the $2s$-restricted eigenvalue
is no smaller than $\delta_{2s}$.
\label{lemma:rip}
\end{proposition}
In other words, although $Q$ has a non-trivial kernel, $Q^TQ$ has a non-zero minimum eigenvalue when operating on a vector $z = (x,E)$ with
 cyclic $2s$-sparse $z_E$, i.e., the following inequality holds if $z = (x,E) \in \R^n
\times \S_{2s}$:
\begin{align}
	0 < \delta_{2s} z^T z &\le z^TQ^TQ z. 
\label{eq:REV}	
\end{align}
\begin{proof}
Inequality \eqref{eq:REV} follows directly from Theorem
\ref{th:injectivity} which states that $Qz\ne 0$ for any $z=(x,E)\ne 0$ with $z_E$ cyclic $2s$-sparse. This  is
equivalent to ${z^TQ^TQz}> 0$. 
Define $\Gamma_{2s} = \textrm{supp}(z_E)$ and $\Gamma = n + \Gamma_{2s}$. Now note that ${z^TQ^TQ z} = {z^T_{\overline{\Gamma}} Q_{\overline{\Gamma}_{2s}}^T Q_{\overline{\Gamma}_{2s}}  z_{\overline{\Gamma}}}$ and hence:
 $${\lambda_{\min}\left\{Q_{\overline{\Gamma}_{2s}}^T Q_{\overline{\Gamma}_{2s}} \right\}z_{\overline{\Gamma}}^T  z_{\overline{\Gamma}}}
\le   {z_{\overline{\Gamma}}^T Q_{\overline{\Gamma}_{2s}}^T Q_{\overline{\Gamma}_{2s}} z_{\overline{\Gamma}}},$$ 
from which we conclude that
inequality \eqref{eq:REV} holds with:
$$\delta_{2s}= \min_{\Gamma_{2s}}  \quad \lambda_{\min}\left\{Q_{\overline{\Gamma}_{2s}}^T Q_{\overline{\Gamma}_{2s}}\right\},$$
where $\Gamma_{2s}$ is defined in \eqref{eq:gamma_def}.
\end{proof}

Similarly to $2s$-sparse observability, the computation of $\delta_{2s}$ is
combinatorial in nature since one needs to calculate the eigenvalues of 
$Q_{\overline{\Gamma}_{2s}}^T Q_{\overline{\Gamma}_{2s}}$ 
for all the different 
sets $\Gamma_{2s}$.

\section{Event Triggered Projected Gradient Descent}
\label{sec:tgd}

The problem of reconstructing a sparse signal from measurements is well studied
in the compressive sensing literature \cite{CandesCS,TaoCS}. In this section we
present an algorithm for state reconstruction that can be seen as an extension 
of the iterative hard thresholding algorithm, reported in \cite{AcceleratedIHT}, 
to the case where part of the signal to be reconstructed is sparse and part is
governed by linear dynamics. We also draw inspiration from \cite{algorithmsAsControlBook}  where it is shown
how to interpret optimization algorithms as dynamical systems with feedback. 
Once this link is established, Lyapunov analysis techniques become available  to
design as well as to analyze the performance of optimization algorithms.

\subsection{The Algorithm}

The Event Triggered Projected Gradient Descent (ETPG) algorithm updates the estimate \mbox{$\hat{z}\in \R^n\times \R^{p\tau}$} of \mbox{$z\in \R^n\times \R^{p\tau}$} by following the gradient direction of the Lyapunov candidate function \mbox{$V(\hat{z}) = \half\norm{Y - Q \hat{z}}^2$}, i.e.
$$ \hat{z}^{+} := \hat{z} + \eta Q^T (Y - Q\hat{z}) $$
for some step size $\eta$.
Since $\hat{z}$ will not, in general, satisfy the desired sparsity constraints, gradient steps are alternated with projection steps using the projection operator:
$$\Pi:\R^n\times\R^{p\tau}\to \R^n\times \mathbb{S}_s$$
that takes $\hat{z}\in \R^n\times\R^{p\tau}$ to the closest point in $\R^n\times \mathbb{S}_s$. 
In order to ensure that $V$ decreases along the execution of the algorithm, multiple gradient steps are performed for each projection step.  By monitoring the evolution of $V$, the algorithm can  determine when the decrease in $V$ caused by the gradient steps compensates the increase caused by $\Pi$. This is akin to triggering an event in event-triggered control \cite{eventTriggered}.

Define the error vector as $e = z^* - \hat{z}$ (where $z^{*}$ denotes the solution of Problem \ref{prob:opt}) and note that $e_E = z^*_E - \hat{z}_E$ is, at most, $2s$-sparse whenever $\hat{z}_E$ is at most $s$-sparse. From Theorem \ref{th:injectivity} we know that the intersection of the cyclic $2s$-sparse set $\S_{2s}$ with the kernel of $Q$ is only one point, $e = 0$, which corresponds to $\hat{z} = z^*$. Hence, by driving $V(\hat{z})$ to zero while forcing $\zhat_E$ to be cyclic
$s$-sparse, $\hat{z}$ is guaranteed to converge to $z^*$. 
Formalizing these ideas leads to Algorithm~\ref{alg:etpg}.

\begin{algorithm}[ht]
\caption{Event-Triggered Projected Gradient Descent}
\begin{algorithmic}[1]
\STATE Initialize $k:= 1$, $m:= 0$, $\hat{z}_\Pi^{(0)}: = 0$;
\WHILE{$V\left(\hat{z}_\Pi^{(k-1)}\right) > 0$} \label{line:guard}
	\STATE $\hat{z}^{(k,0)} := \hat{z}_\Pi^{(k-1)}$;
	\STATE reset $m: = 0$, $V_{temp}: = V\left(\hat{z}_\Pi^{(k-1)}\right)$;
	\WHILE {$V_{temp} \ge V\left(\hat{z}_\Pi^{(k-1)}\right)$} \label{line:trigger}
		\STATE $\hat{z}^{(k,m+1)}:= \hat{z}^{(k,m)} + \eta Q^T\left(Y - Q \hat{z}^{(k,m)} \right)$;
		\STATE $V_{temp} := V\left( \Pi(\hat{z}^{(k,m+1)}) \right)$;
		\STATE $m := m+1$;
	\ENDWHILE
	\STATE $\hat{z}_\Pi^{(k)} := \Pi(\hat{z}^{(k,m)})$;
	\STATE $k: = k + 1$;
\ENDWHILE
\STATE \textbf{return} $\hat{z}_\Pi^{(k)}$
\end{algorithmic}
\label{alg:etpg}
\end{algorithm}

A typical execution of the ETPG algorithm is illustrated in Figure \ref{fig:lyap_intro} where the
evolution of the Lyapunov candidate function $V$ is shown. We can see that at
$k=0,1$ the ETPG algorithm applied one gradient step before
applying the new projection step, while at $k=2$ and $k=8$ two gradient steps were 
executed to compensate for the increase in $V$ caused by the projection step. This adaptive nature of the
ETPG algorithm is a consequence of the event-triggering mechanism. The algorithm
ensures that the subsequence depicted in blue in Figure
\ref{fig:lyap_intro} is a converging subsequence and thus 
stability in the Lyapunov sense is attained.

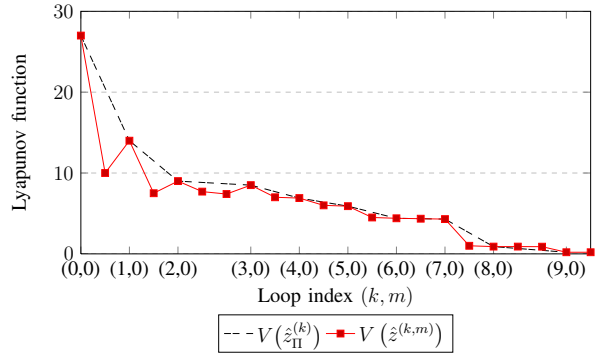
\begin{figure}
\centering
\resizebox{0.48\textwidth}{!}{
\begin{tikzpicture}
		\begin{axis}[
	    	xlabel=Loop index ${(k,m)}$,
		xmin = {(0,0)},
		xmax = {(9,1)},
		ymax = 30,
		ymin = 0,
		y=0.5cm/3,
    		x=0.5cm,
	    	ylabel=Lyapunov function,
		symbolic x coords={ 
						{(0,0)},{(0,1)},
						{(1,0)},{(1,1)},
						{(2,0)},{(2,1)},{(2,2)},
						{(3,0)},{(3,1)},
						{(4,0)},{(4,1)},
						{(5,0)},{(5,1)},
						{(6,0)}, {(6,1)},
						{(7,0)}, {(7,1)},
						{(8,0)},{(8,1)},{(8,2)},
						{(9,0)},{(9,1)}} ,
		xtick={{(0,0)},{(1,0)},{(2,0)},{(3,0)},{(4,0)},{(5,0)},{(6,0)},{(7,0)},{(8,0)},{(9,0)}},
	    	legend style={at={(0.5,-0.25)},anchor=north,legend columns=-1},
	    	ymajorgrids=true,
	    	grid style=dashed,
		]
		\addplot [dash pattern=on 4pt off 1pt on 4pt off 4pt] 
		coordinates {	({(0,0)}, 27) ({(1,0)}, 14) ({(2,0)}, 9) ({(3,0)}, 8.5) ({(4,0)}, 6.9) ({(5,0)}, 5.9) ({(6,0)}, 4.4) ({(7,0)}, 4.3)   ({(8,0)}, 0.9) ({(9,0)}, 0.2)
				};	
		\addplot 
			coordinates {	({(0,0)}, 27) ({(0,1)}, 10) 
						({(1,0)}, 14) ({(1,1)}, 7.5) 
						({(2,0)}, 9) ({(2,1)}, 7.7) ({(2,2)}, 7.4) 
						({(3,0)}, 8.5) ({(3,1)}, 7) 
						({(4,0)}, 6.9) ({(4,1)}, 6) 
						({(5,0)}, 5.9) ({(5,1)}, 4.5) 
						({(6,0)}, 4.4) ({(6,1)}, 4.35) 
						({(7,0)}, 4.3) ({(7,1)}, 1) 
						({(8,0)}, 0.9) ({(8,1)}, 0.9) ({(8,2)}, 0.89) 
						({(9,0)}, 0.2) ({(9,1)}, 0.2)
					};			
		\legend{$V\big(\hat{z}_\Pi^{(k)}\big)$, $V\left(\hat{z}^{(k,m)}\right)$}	
		\end{axis}
	\end{tikzpicture}}
	\caption{An example of the evolution of the Lyapunov candidate function V while
	running the ETPG algorithm. The subsequence highlighted in blue (dashed) is a decreasing subsequence.}
	\label{fig:lyap_intro}
\end{figure}

The ETPG algorithm uses two counters, one for each loop. Accordingly, we will
use the notation $\hat{z}^{(k,m)}$ where $k$ counts the number of iterations of the
outer loop, while $m$ counts the number of internal gradient
descent steps within each iteration of the inner loop.

\subsection{The Projection Operator}

Before discussing the convergence of the proposed algorithm, it is important to show how to compute the projection operator $\Pi$.
\begin{definition}
Given a vector $z = (x,E) \in \R^n \times \R^{p\tau}$, we denote by $\Pi(z)$ 
the element of $\R^n \times \S_s$ closest to $z$ in the 2-norm sense, i.e.,
\begin{align}
\norm{\Pi(z) - z} \le \norm{z' - z},
\label{eq:triangEq}
\end{align}
for any $z' \in \R^n \times \S_s$.
\label{def:projection}
\end{definition}

We first note that $\Pi(z)=\Pi(x,E)=(x,\Pi'(E))$. To explain how $\Pi'$ is computed, recall from Example \ref{ex:example1} the matrix $\tilde{E}$ obtained by formatting $E \in \R^{p\tau}$ so that the $i$th column of $\tilde{E}$ is given by $a(t-i+1)$.
Now, define $\overline{E} \in \R^p$ by $\overline{E}_i =
\norm{\tilde{E}_i}^2$. By noting that $\norm{E}^2 = \norm{\overline{E}}^2$ and that $E$ is cyclic $s$-sparse if and only if $\overline{E}$ is $s$-sparse, it immediately follows that $\Pi'(z)$ can be computed by setting to zero the elements of $E$ corresponding to the smallest $p - s$ entries of $\overline{E}$.
\begin{example}
Let us consider the following example with $p = 3, \tau = 2$ and $s=1$:
\begin{align*}
	E = \matrix{1 & 2 & 3 & 4 & 5 & 6 & 7 & 8 & 9}^T \in \R^{p\tau}
\end{align*}
Hence,
\begin{align*}
	\tilde{E} &= \matrix{1 & 4 & 7 \\ 2 & 5 & 8 \\ 3 & 6 & 9} \in \R^{p\times \tau}, \qquad
	 \overline{E} = \matrix{66 \\ 93 \\ 126} \in \R^{p}, \qquad
	 \Pi'(\overline{E} ) = \matrix{0 \\ 0 \\ 126} \in \R^{p},
\end{align*}
which leads to:
\begin{align*}
	\Pi'(E) = \matrix{0 & 0 & 3 & 0 & 0 & 6 & 0 & 0 & 9}^T.
\end{align*}
\end{example}

\subsection{Convergence of the ETPG Algorithm}
\label{sec:stability}

In this subsection we discuss the convergence and performance of the ETPG
algorithm. 
The main result of this section is stated in the following theorem.

\begin{theorem}{\textbf{(Convergence of the ETPG algorithm)}}
Let the linear control system defined by~(\ref{eq:sys_state})
and \eqref{eq:sys_out} be $2s$-sparse observable with $2s$-restricted eigenvalue $\delta_{2s}$ and let  $z^* = (x^*, E^*) \in
\R^n \times \S_s$ denote the solution of Problem \ref{prob:opt}. The ETPG
algorithm converges to $z^*$ if the following holds:
\begin{enumerate}
\item $0 < \eta < 2\lambda^{-1}_{\max}\{Q^T Q\}$,
\item 
$\delta_{2s} > \frac{4}{9}{\lambda_{\max}\{Q^TQ\}}$.
\end{enumerate}
\label{th:stability}	
\end{theorem}
\begin{remark}
Theorem \ref{th:stability} shows that the ETPG algorithm correctly reconstructs the state whenever the attacker has access to no more than $s$ sensors and the system is $2s$-sparse observable. In practice, one does not exactly know the number $s$ of attacked sensors. However, if an upper bound $\overline{s}$ is known, the ETPG is guaranteed to work as long as the system is $2\overline{s}$-sparse observable. For this reason the ETPG algorithm is run with the largest $s$ for which $2s$-spare observability holds.
\end{remark}
\begin{remark}
The ETPG algorithm can be seen as a generalization of the ``Normalized Iterative Hard Thresholding'' (NIHT) algorithm presented in \cite{NormlaizedIHT}. The ETPG algorithm uses multiple gradient descent steps supervised by an event-triggering mechanism whereas the NIHT algorithm only uses one gradient step.
These two ingredients, multiple gradient steps and event-triggering, result in weaker conditions for convergence: $\delta_{2s} > \frac{4}{9}{\lambda_{\max}\{Q^TQ\}}$ for the ETPG algorithm vs $\delta_{2s} > \frac{8}{9}{\lambda_{\max}\{Q^TQ\}}$ for the NIHT algorithm (Theorem 4 in \cite{NormlaizedIHT}).
\end{remark}

In the remainder of this subsection, we focus on proving this theorem by resorting to the Lyapunov candidate function:
\begin{align*}
	W\left(\hat{z}\right) & = \norm{z^{*} - \hat{z}} = \norm{e}.
\end{align*}
Note that, unlike $V$, the Lyapunov candidate function $W$ has a unique minimum at $\hat{z} = z^{*}$. However, one
should note that evaluation of $W$ requires the knowledge of $z^*$ which we do
not have a priori. Accordingly, $W$  is only adequate  to discuss stability but can not be used to design the ETPG algorithm.

The proof is divided into small steps. First we discuss the effect
of the two operations used inside the ETPG algorithm: projection
and gradient descent. This is done in Propositions
\ref{prop:sparsify}, \ref{prop:angle}, and \ref{prop:gradient}. Next we 
provide sufficient conditions under which the inner loop is guaranteed to
terminate. These are given in Proposition \ref{prop:termination}. Finally, we use the triggering condition to establish that $W$ is a Lyapunov function.

\subsubsection{Effect of the Projection Operator}
\begin{proposition}
The following inequality holds for any $z \in \R^{n} \times \R^{p\tau}$:
\begin{align*}
	W \circ \Pi(z) &\le 2 W(z).
\end{align*}
\label{prop:sparsify}
\end{proposition} 
\begin{proof} The result is proved by direct computation as
follows:
\begin{align*}	
	W \circ \Pi(z)	&= \norm{z^* - \Pi(z) }
				 \le \norm{z^* - z} + \norm{\Pi(z) - z}
					 \le 2 \norm{z^* - z}
					 = 2 W(z),
\end{align*}
where the first inequality follows from the triangular inequality and the second inequality follows from ~\eqref{eq:triangEq}.
\end{proof}

\subsubsection{Effect of Gradient Steps}

In this subsection we study the effect of the gradient descent step along with sufficient conditions for termination of the inner loop of the ETPG algorithm. The first proposition is a technical result used in the proof of the second proposition that characterizes the decease in $W$ caused by gradient descent steps.

\begin{proposition}
Let the linear control system defined by~(\ref{eq:sys_state}) and~(\ref{eq:sys_out}) be $2s$-sparse observable with $2s$-restricted eigenvalue $\delta_{2s}$. The following holds for any $z\in \R^{n} \times \S_{2s}$:
	$$ \norm{(I - Q^+Q) z}  \le (1 - \delta_{2s} \lambda_{\max}^{-1}\{Q^T Q\})^\half \norm{z}$$
where $Q^+$ is the Moore-Penrose pseudo inverse of Q.
\label{prop:angle}
\end{proposition}
\begin{proof}
The result follows from direct computations as follows:
\begin{align*}
\norm{(I - Q^+Q) z }^2 &= z^T (I - Q^+Q)^2 z 
					\stackrel{(a)}{=}z^T (I - Q^+Q) z
					\stackrel{(b)}{=}z^T (I - Q^T(QQ^T)^{-1}Q) z\\
					&\stackrel{(c)}{\le}z^T (I - \lambda_{\max}^{-1}\{Q^T Q\}Q^TQ)z 
					\stackrel{(d)}{\le} (1 - \delta_{2s}\lambda_{\max}^{-1}\{Q^T Q\}) {z^Tz}
\end{align*}
where equality: $(a)$ follows by noticing that the projection operator $I - Q^+Q$ is idempotent; $(b)$ follows from the definition of $Q^+ = Q^T (QQ^T)^{-1}$; $(c)$ follows from the fact that $QQ^T = I + \O\O^T$ is a positive definite matrix and hence the inverse $(QQ^T)^{-1}$ exists and can be bounded as $(QQ^T)^{-1} \ge \lambda_{\max}^{-1}\{Q^T Q\} I$; and $(d)$ follow from the $2s$-sparse observability assumption along with Proposition \ref{lemma:rip}.
\end{proof}

\begin{proposition}
Let the linear control system defined by~(\ref{eq:sys_state})
and~(\ref{eq:sys_out}) be $2s$-sparse observable with $2s$-restricted eigenvalue $\delta_{2s}$. If the following conditions hold:
\begin{enumerate}
\item the estimate $\zhat_E^{(k,0)}$ is $s$-sparse, i.e., $\zhat^{(k,0)} = \left(\zhat_x^{(k,0)},\zhat_E^{(k,0)}\right) \in \R^n \times \S_s$, 
\item the step size $\eta$ satisfies $\eta < 2 \lambda_{\max}^{-1}\{Q^T Q\}$,
\end{enumerate}
then for any $\varepsilon \in \R^+$, the following inequality holds for any $m\ge  \left \lceil \frac{\log \varepsilon}{\log \left(1 - \eta \delta_{2s} \right)} \right \rceil$
\begin{align}
	W\left(\zhat^{(k,m)}\right) &\le \left( \left(1 - \delta_{2s} \lambda_{\max}^{-1}\{Q^T Q\} \right)^\half+ \varepsilon \right) W\left(\hat{z}^{(k,0)}\right),
\label{eq:decreaseWpropGradient}	
\end{align}
where $\hat{z}^{(k,m)}\in \R^{n} \times \R^{p \tau}$ is recursively defined by:
$$\hat{z}^{(k,m+1)} := \hat{z}^{(k,m)} + \eta Q^T \left(Y - Q \hat{z}^{(k,m)}\right).$$
\label{prop:gradient}
\end{proposition}
\begin{proof} The Lyapunov candidate function $W\left(\hat{z}^{(k,m)}\right)$, after taking
one gradient descent step, can be written as follows:
\begin{align*}
	W\left(\hat{z}^{(k,m)}\right) 	&= \norm{z^* - \hat{z}^{(k,m)}}
						= \norm{z^* - \hat{z}^{(k,m-1)} - \eta Q^T \left(Y - Q \hat{z}^{(k,m-1)}\right)} \\
						&= \norm{z^* - \hat{z}^{(k,m-1)} - \eta  Q^T \left(Q z^* - Q \hat{z}^{(k,m-1)} \right)}\\
						&= \norm{e^{(k,m-1)} - \eta  Q^T Q e^{(k,m-1)}}.
\end{align*}
Recursively extending the previous analysis  to $m$ steps we obtain:
\begin{align}					
	W\left(\hat{z}^{(k,m)}\right)	&= \norm{ (I - \eta Q^T Q)^m  e^{(k,0)} }.
\label{eq:Wintermediate}	
\end{align}	
Now define the projection error $\epsilon(m)$ as:
\begin{align}
\epsilon(m) = (I - Q^+Q) - (I - \eta Q^T Q)^m 
\label{eq:epsilonInterm}
\end{align}	 
where $Q^+$ is the Moore-Penrose pseudo inverse of $Q$.
It follows from Corollary 3 in \cite{Petryshyn1967417} that:
$$ \norm{\epsilon(m)} \le \left(1 - \eta \delta_{2s} \right)^m.$$
Given any $\varepsilon \in \R^{+}$, we conclude that in no more than:
$$ \left \lceil \frac{\log \varepsilon}{\log \left(1 - \eta \delta_{2s} \right)} \right \rceil$$
steps, the projection error $\norm{\epsilon(m)}$ satisfies $\norm{\epsilon(m)} \le \varepsilon$.
Hence:
\begin{align*}	
W\left(\hat{z}^{(k,m)}\right)
			&\stackrel{(a)}{=}  \norm{ (I - Q^+Q) e^{(k,0)} - \epsilon(m) e^{(k,0)}} \\
			&\stackrel{(b)}{\le} \norm{(I - Q^+Q) e^{(k,0)}} + \norm{\epsilon(m) e^{(k,0)}}\\
			&\stackrel{(c)}{\le} \left(\left(1 - \delta_{2s} \lambda_{\max}^{-1}\{Q^T Q\} \right)^\half+ \norm{\epsilon(m)}\right) \norm{e^{(k,0)}}\\
			&\le \left( \left(1 - \delta_{2s} \lambda_{\max}^{-1}\{Q^T Q\} \right)^\half+ \varepsilon \right) W\left(\hat{z}^{(k,0)}\right)
\end{align*}
where the equality: $(a)$ follows by substituting~\eqref{eq:epsilonInterm} in~\eqref{eq:Wintermediate}; $(b)$ follows from the triangular inequality; and $(c)$ follows from Proposition \ref{prop:angle} and the first assumption which guarantees that $\zhat_E^{(k,0)}$ is $s$-sparse and hence $z^*_E-\zhat_E^{(k,0)}$ is at most $2s$-sparse. 
\end{proof}
\begin{remark}
It follows from the previous analysis that we can replace the inner loop in Algorithm \ref{alg:etpg} with a one step projection of $\hat{z}$ on the kernel of $Q$, that is, Algorithm \ref{alg:etpg} can be simplified to the alternation between the following two steps:
\begin{align*}
\hat{z}^{(k+1)} &:= \hat{z}_\Pi^{(k)} + Q^+\left(Y - Q\hat{z}_\Pi^{(k)} \right),\\
\hat{z}_\Pi^{(k+1)} &:= \Pi\left(\hat{z}^{(k+1)}\right).
\end{align*}
However, since computing the pseudo inverse matrix $Q^+$ can, in general, suffer from numerical issues, we argue that using the gradient descent algorithm (or any other recursive implementation that computes $Q^+$, e.g., Newton method,  conjugate gradients, ...) is preferable.
\label{rem:2}
\end{remark}

\subsubsection{Termination of the ETPG's Inner Loop}
In the following result, we establish sufficient conditions for termination of the inner loop.
\begin{proposition}
Let the linear control system defined by~(\ref{eq:sys_state})
and~(\ref{eq:sys_out}) be $2s$-sparse observable with $2s$-restricted eigenvalue $\delta_{2s}$. If the following holds:
\begin{enumerate}
\item $0 < \eta < 2 \lambda_{\max}^{-1}\{Q^T Q\}$,
\item 
$\delta_{2s} > \frac{4}{9}{\lambda_{\max}\{Q^TQ\}}$
\end{enumerate}
then, after no more than:
\begin{equation}
\label{Eq:MSteps}
\left \lceil \frac{\log \frac{3}{2}\left(\delta_{2s}{\lambda_{\max}^{-1}\left\{Q^TQ\right\}}\right)^\half - 1}{\log \left(1 - \eta \delta_{2s} \right)} \right \rceil
\end{equation}
inner loop iterations, the inner loop condition $V\left(\hat{z}_\Pi^{(k)}\right) < V\left(\hat{z}_\Pi^{(k-1)}\right)$ is satisfied.
\label{prop:termination}
\end{proposition}
\begin{proof}
Before we start, we recall that $V\big(\hat{z}^{(k)}_\Pi\big) = V\left(\hat{z}^{(k,0)}\right)= \half \norm{Y - Q\hat{z}^{(k,0)}}^2 = \half \norm{Q e^{(k,0)}}^2$. It follows from the definition of $e^{(k,0)} = z^* - \Pi(z^{(k-1,m)})$ that $e^{(k,0)}_E$ is at most $2s$-sparse. Accordingly, inequality \eqref{eq:REV} holds as follows:
\begin{align}
\frac{\delta_{2s}}{2} W^2\left(\hat{z}^{(k)}_\Pi\right) \le V\left(\hat{z}^{(k)}_\Pi \right) \le \frac{\lambda_{\max}\{Q^TQ\} }{2} W^2\left(\hat{z}^{(k)}_\Pi \right)
\label{eq:V_W}
\end{align}
Hence, if we can prove that applying $\left \lceil \frac{\log \frac{3}{2}\left(\delta_{2s}{\lambda_{\max}^{-1}\left\{Q^TQ\right\}}\right)^\half - 1}{\log \left(1 - \eta \delta_{2s} \right)} \right \rceil$ gradient descent steps followed by a projection step implies that:
\begin{align}
 W^2\left(\hat{z}_\Pi^{(k)}\right) < \delta_{2s} \lambda_{\max}^{-1}\{Q^TQ\}  W^2\left(\hat{z}_\Pi^{(k-1)}\right)
\label{eq:W_W}
\end{align}
holds, we can combine \eqref{eq:W_W} with \eqref{eq:V_W} to obtain:
$$ V\left(\hat{z}_\Pi^{(k)}\right) \le \frac{\lambda_{\max}\{Q^TQ\}}{2} W^2\left(\hat{z}_\Pi^{(k)}\right) < \frac{\delta_{2s}}{2}  W^2\left(\hat{z}_\Pi^{(k-1)}\right) \le V\left(\hat{z}_\Pi^{(k-1)}\right) $$
and conclude that the inner loop condition is satisfied.
Therefore, to finalize the proof we need to show that inequality \eqref{eq:W_W} holds. This follows directly from:
\begin{align*}
W\left(\hat{z}_\Pi^{(k)}\right) &=
W\left(\hat{z}^{(k,0)}\right) 
			= \norm{z^* - \Pi(\hat{z}^{(k-1,m)})} 
			\stackrel{(a)}{\le} 2W\left(\hat{z}^{(k-1,m)}\right)
			\\&\stackrel{(b)}{\le} 2\left( \left(1 - \delta_{2s} \lambda_{\max}^{-1}\{Q^T Q\} \right)^\half+ \varepsilon \right) W\left(\hat{z}_\Pi^{(k-1)}\right)
\end{align*}
where the inequality $(a)$ follows from Proposition \ref{prop:sparsify} while inequality $(b)$ follows from Proposition \ref{prop:gradient}. To conclude the result, we need to show that the factor $2\left( \left(1 - \delta_{2s} \lambda_{\max}^{-1}\{Q^T Q\} \right)^\half+ \varepsilon \right)$ is strictly less than $(\delta_{2s} \lambda_{\max}^{-1}\{Q^TQ\})^\half$. But this follows directly from assumption 2):
\begin{align*}
\delta_{2s}{\lambda_{\max}^{-1}\left\{Q^TQ\right\}} > \frac{4}{9}
& \Rightarrow 3 \left(\delta_{2s}{\lambda_{\max}^{-1}\left\{Q^TQ\right\}}\right)^\half > 2
\\& \Rightarrow 2 \left(1 -  \left(\delta_{2s}{\lambda_{\max}^{-1}\left\{Q^TQ\right\}}\right)^\half \right) < \left(\delta_{2s}{\lambda_{\max}^{-1}\left\{Q^TQ\right\}}\right)^\half
\\& \Rightarrow 2 \left(1 -  \left(\delta_{2s}{\lambda_{\max}^{-1}\left\{Q^TQ\right\}}\right)^\half \right) + 2 \varepsilon  < \left(\delta_{2s}{\lambda_{\max}^{-1}\left\{Q^TQ\right\}}\right)^\half
\end{align*}
for any $\varepsilon$ satisfying: 
$$2 \varepsilon <  \left(\delta_{2s}{\lambda_{\max}^{-1}\left\{Q^TQ\right\}}\right)^\half - 2 \left(1 -  \left(\delta_{2s}{\lambda_{\max}^{-1}\left\{Q^TQ\right\}}\right)^\half \right) = 3 \left(\delta_{2s}{\lambda_{\max}^{-1}\left\{Q^TQ\right\}}\right)^\half - 2.$$
\end{proof}

\subsection{Stability of ETPG Algorithm}
Using the previous results, we can now show  convergence of the ETPG algorithm.
\begin{proof}[Proof of Theorem \ref{th:stability}]
Convergence of the algorithm, in the Lyapunov sense, follows directly from the termination of the inner loop shown in Proposition \ref{prop:termination}. Consider the following sequence:
\begin{align*}
Y^{(k,m)} = \sqrt{\frac{2}{\delta_{2s}}V\left(\zhat^{(k,0)}\right)} = \sqrt{\frac{2}{\delta_{2s}} \half \norm{Q e^{(k,0)}}^2}
\end{align*}

We first show that this sequence forms an upper bound for $W\left(\hat{z}^{(k,m)}\right)$.
For $m = 0$, it follows from the definition of $e^{(k,0)} = z^* - \hat{z}^{(k,0)} = z^* - \Pi(\hat{z}^{(k-1,m)})$ that $e_E^{(k,0)}$ is at most $2s$-sparse. Accordingly inequality~\eqref{eq:V_W} holds as follows:
\begin{align}
	\frac{\delta_{2s}}{2} W^2\left(\hat{z}^{(k,0)}\right) \le V\left(\hat{z}^{(k,0)}\right) \Rightarrow W\left(\hat{z}^{(k,0)}\right) \le Y^{(k,0)}
	\label{eq:bounds0}
\end{align}
Note that upper bound \eqref{eq:bounds0} is satisfied only when $\ahat$ is cyclic $s$-sparse. Thus it is only applicable after applying the projection operator, i.e., when $m=0$.
We now extend this bound for $m > 0$ as follows:
\begin{align}
	W\left(\zhat^{(k,m)}\right) < W\left(\zhat^{(k,0)}\right) = Y^{(k,0)} = Y^{(k,m)},
\label{eq:boundsm}	
\end{align}
where the inequality follows from Proposition \ref{prop:gradient} along with assumption 2).

 Since $0 \le W\left(\zhat^{(k,m)}\right) \le Y^{(k,m)}$, and by the triggering condition $V\left({\hat{z}_\Pi^{(k+1)}}\right) < V\left(\hat{z}_\Pi^{(k)}\right)$ we have:
  $$\lim_{k \rightarrow \infty} Y^{(k,m)} = \lim_{k \rightarrow \infty}
\sqrt{\frac{2}{\delta_{2s}} V\left(\zhat^{(k,0)}\right)} = 0$$
 and we conclude:
 $$ \lim_{k \rightarrow
\infty} W\left(\zhat^{(k,0)}\right) = 0.$$
\end{proof}

\section{An Event-Triggered Projected Luenberger Observer}
\label{sec:luenberger}

In this section we describe a solution to Problem \ref{prob:observer} obtained by rendering the ETPG algorithm recursive.
\subsection{The Algorithm}
We start again with the linear dynamical system defined by~(\ref{eq:sys_state})
and~(\ref{eq:sys_out}). The dynamics of $z(t)
=(x(t-\tau+1), E(t)) \in \R^{n} \times \R^{p\tau}$ can be written, using the equality $a(t)=y(t)-Cx(t)$, as follows:
\begin{align*}
	\matrix{
	x(t-\tau+1) \\ a(t-\tau+1) \\ \vdots \\ a(t)
	}
	 &= 
	\matrix{A & 0 & 0 & \hdots & 0\\ 0 & 0 & I & \hdots & 0 \\ \vdots \\ -C A^{\tau} & 0 & 0
	& \hdots & 0}
	\matrix{
	x(t-\tau) \\ a(t-\tau) \\ \vdots \\ a(t-1)
	} + \\
	&
	\matrix{
	B & 0 & \hdots & 0 \\ 0 & 0 & \hdots & 0  \\ \vdots & &
	& 
	\vdots
	\\
	-CA^{\tau-3}B & -CA^{\tau-2}B & \hdots & -CB }
	\matrix{ 
	u(t-\tau) \\ u(t-\tau+1) \\ \vdots \\ u(t-1)
	} + \matrix{0 \\ 0 \\ \vdots \\ I} y(t) ,
\end{align*}
\begin{align*}	
	\matrix{
	y(t-\tau) \\ y(t-\tau+1) \\ \vdots \\ y(t-1)
	} = & 
	\matrix{
	C \\ CA \\ \vdots \\ CA^{\tau}
	}
	x(t-\tau)
	+ 
	\matrix{
	a(t-\tau) \\ a(t-\tau+1) \\ \vdots \\ a(t-1)
	} + F U(t-1)
\end{align*}
or in the more compact form:
\begin{align*}
	z(t)   &= \overline{A} z(t-1) + \overline{B} \bar{u}(t-1) \\
	Y(t-1) 	&= Q z(t-1),
\end{align*}
where:
\begin{align*}
	\overline{A}&=
	\left[ \begin{array}{cccc}
	A & 0 &  \hdots & 0\\ \hline 0 & I & \hdots & 0 \\ 
	\vdots & & \ddots &\\
	0 & 0 & \hdots & I \\
	-C A^{\tau} & 0 & \hdots & 0
	\end{array} \right],
	\overline{B} &= \left[ \begin{array}{cccc|c}
	B & 0 & \hdots & 0 & 0 \\ \hline  0 & 0 & \hdots & 0 & 0  \\ \vdots & &
	& 
	\vdots & \vdots
	\\
	0 & 0 & \hdots & 0 & 0\\
	-CA^{\tau-3}B & -CA^{\tau-2}B & \hdots & -CB & I 
	\end{array} \right],
\end{align*}
$Y(t-1) = \tilde{Y}(t-1) - FU(t-1)$, $\bar{u}(t-1) = \matrix{ U(t-1) \\
y(t) }$, and $z(t), \tilde{Y}(t), U(t), Q, F$ are as defined as in Section \ref{sec:problem}.

The Event-Triggered Projected Luenberger (ETPL) Observer consist of the iteration of the following two steps:

\paragraph{Time Update}
Starting from an estimate $\zhat_\Pi(t-1) = (\xhat(t-1),\ahat_\Pi(t-1))$ with $\ahat_\Pi(t-1)$ $s$-sparse, we use the dynamics to update the previous estimate: 
\begin{align*}
	\zhat(t) = \overline{A} \zhat_\Pi(t-1) + \overline{B} \overline{u}(t-1).
\end{align*}
This may result in an increase in the value of the Lyapunov candidate function $V$.

\paragraph{Event-Triggered Projected Luenberger (Measurement) Update}
In this step, we alternate between applying the Luenberger update step:
\begin{align*}
	\zhat^{(m+1)}(t) &= \zhat^{(m)}(t) + L \left(Y(t) - Q \zhat^{(m)}(t) \right),
\end{align*}
multiple times followed by the projection operator $\Pi$ once:
\begin{align*}
	\zhat_\Pi(t) &= \Pi \left( \zhat^{(m)}(t)\right).
\end{align*}
It follows from the proof of Theorem~\ref{th:stability} that alternating between multiple Luenberger updates (which is the generalization of the gradient descent step when $L = Q^T \Sigma$ for some positive definite matrix $\Sigma$) and projection steps forces a decrease of the Lyapunov candidate function $V(\hat{z}(t)) = \half \norm{Y(t) - Q \hat{z}(t)}^2$. In order to compensate for the increase introduced by the time update step, we need to ensure that $V$ decreases along the execution of the algorithm. Hence, multiple Luenberger updates and projection steps are performed for each time-update step.  Using the same triggering technique, by monitoring the evolution of $V$, the algorithm can  determine when the decrease in $V$ caused by the Luenberger update/projection steps compensates the increase caused by the time-update. 
This sequence of steps results in Algorithm \ref{alg:etpl}.

\begin{algorithm}
\caption{Event Triggered Projected Luenberger Observer}
\begin{algorithmic}
\STATE \textbf{a)} Time Update:
\STATE $\zhat(t) = \overline{A} \zhat(t-1) + \overline{B} \overline{u}(t-1)$;
\STATE \textbf{b)} Event-Triggered Projected Luenberger (Measurement) Update: 
\WHILE{$V(\zhat_\Pi(t)) \ge V(\zhat_\Pi(t-1))$}
\STATE reset $m: = 0$, $\zhat^{(0)}(t) = \zhat_\Pi(t) =  \Pi(\zhat(t))$;
\STATE $V_{temp}: = V(\zhat_\Pi(t))$;
\WHILE {$V_{temp} \ge V(\zhat_\Pi(t))$}	
	\STATE $\zhat^{(m+1)}(t) := \zhat^{(m)}(t) + L \left(Y(t) - Q \zhat^{(m)}(t) \right)$;
	\STATE $V_{temp} := V(\Pi \left(\zhat^{(m+1)}(t) \right))$;
	\STATE $m := m+1$;
\ENDWHILE
\STATE $\zhat(t) := \zhat^{(m)}(t)$;
\ENDWHILE
\end{algorithmic}
\label{alg:etpl}
\end{algorithm}

To make the connection with the standard Luenberger observer clearer, assume that only one Luenberger/projection update is required per time update. In this case, the ETPL observer can be written as:
\begin{align*}
	\zhat(t) &= \overline{A} \zhat_\Pi(t-1) + \overline{B} \overline{u}(t-1) \\&\qquad \qquad \qquad + L'\left (Y(t-1) - Q \zhat_\Pi(t-1)\right)\\
	\zhat_\Pi(t) &= \Pi ( \zhat(t) )
\end{align*}
which has the form of a standard Luenberger observer with gain $L' = A \; L$ along with the projection $\Pi$ step.



\begin{remark}
It follows from Remark \ref{rem:2} that we can replace the inner loop in Algorithm \ref{alg:etpl} with one update step if we fix the Luenberger gain $L$ to be $L = Q^{+}$.
\end{remark}

\subsection{Convergence of the ETPL Observer}
In this subsection we discuss the convergence and performance of the ETPL observer. The main result of this subsection is stated in the following theorem.
\begin{theorem}{\textbf{(Convergence of the ETPL observer)}}
Let the linear control system defined by \\ \eqref{eq:sys_state}
and \eqref{eq:sys_out} be 2s-observable system with $2s$-restricted eigenvalue
$\delta_{2s}$. If the following condition holds:
$$\delta_{2s} > \frac{4}{9}{\lambda_{\max}\{Q^TQ\}}$$ 
then the dynamical system defined by the ETPL observer is a solution to Problem \ref{prob:observer} whenever $L = Q^T \Sigma$ for any positive definite weighting matrix  $\Sigma$ satisfying $\lambda_{\max} \{ \Sigma\} < \lambda_{\max}^{-1}\{Q^TQ \}$. 
\label{th:Luenberger}	
\end{theorem}

The proof of Theorem \ref{th:Luenberger} is based on the Lyapunov candidate function:
\begin{align*}
	W(t) = \norm{z(t) - \zhat(t)} = \norm{e(t)},
\end{align*}
and follows the exact same argument used in the proof of Theorem \ref{th:stability} with the exception of the need to account for an increase in $V$ caused by the time update step. Such increase is compensated by applying the Luenberger update loop multiple times.  This increase is described in the next result.

\subsection{Effect of Time Update}
\begin{proposition}
Consider the linear control system defined by~(\ref{eq:sys_state})
and~(\ref{eq:sys_out}). The following inequality holds:
\begin{align*}
	W(\zhat(t)) &\le \norm{\overline{A}} W(\zhat(t-1)),
\end{align*}
whenever $\hat{z}(t)\in \R^{n} \times \R^{p \tau}$ and $\hat{z}(t-1)\in \R^{n} \times \R^{p \tau}$ are related by: 
$$\zhat(t) = \overline{A} \zhat(t-1) + \overline{B} \overline{u}(t-1).$$
\label{prop:timeupdate}
\end{proposition}

\begin{proof}
The Lyapunov candidate function $W(\zhat(t))$ after applying the time update
step, can be written as follows:
\begin{align*}
	W(\zhat(t)) &= \norm{z^*(t) - \zhat(t)}
				= \norm{\overline{A} z^*(t-1) + \overline{B} \overline{u}(t-1) -
				\overline{A} \zhat(t-1) - \overline{B}
				\overline{u}(t-1)} \\
				&= \norm{\overline{A} (z^*(t-1) - \zhat(t-1))}
				\le
				\norm{\overline{A}} W(\zhat(t-1)).
\end{align*}
\end{proof}

\begin{proof}[Proof of Theorem \ref{th:Luenberger}]
First, we note that the Luenberger update loop in Algorithm \ref{alg:etpl} is identical to the outer loop of Algorithm \ref{alg:etpg} with the loop guard $V\left(\hat{z}_\Pi^{(k)}\right) < 0 $ (Line \ref{line:guard} in Algorithm \ref{alg:etpg}) replaced with $V\left(\hat{z}_\Pi{(t)}\right) < V\left(\hat{z}_\Pi{(t-1)}\right)$. Hence, it follows from Theorem \ref{th:stability} that ``Event-Triggered Luenberger Update'' loop terminates. From the termination of the Luenberger update loop, we conclude that the increase caused by the time update (Proposition \ref{prop:timeupdate}) can always be compensated. Accordingly, using the same argument that was used in the proof of Theorem \ref{th:stability}, we conclude that $\lim_{t \rightarrow \infty} W\left(\zhat_\Pi{(t)}\right) = 0$ and hence the estimate converges to the desired value $z^*$.
\end{proof}

\section{Simulation Results}
\label{sec:results}

\subsection{Multiple Attacking Sequences}
In this example
we consider an Unmanned Ground Vehicle (UGV) under different types of sensor
attacks.  We assume that the UGV moves along straight lines and completely stops before
rotating. Under these assumptions, we can describe the dynamics of the UGV by:

\begin{align*}
	\matrix{\dot{x} \\ \dot{v} } &= \matrix{0 & 1  \\ 0 & \frac{-B}{M} } \matrix{x
	\\ v  } + \matrix{0  \\ \frac{1}{M}  } F,\\
	\matrix{\dot{\theta} \\ \dot{\omega} } &= \matrix{0 & 1 \\ 0 & \frac{-B_r}{J}}
	\matrix{\theta \\ \omega } + \matrix{0 \\ \frac{1}{J} } T,
\end{align*}
where $x,v,\theta,\omega$ are the states of the UGV corresponding to position,
linear velocity, angular position and angular velocity, respectively. The parameters $M,J,B,B_r$  denote the mechanical mass, inertia, translational 
friction coefficient and the rotational friction coefficient. The inputs to 
the UGV are the force $F$ and the torque $T$.  The UGV is equipped with a GPS sensor which measures the UGV position, two motor
encoders which measure the translational velocity and an inertial measurement unit which measures
both rotational velocity and rotational position. We assume that encoders perform the necessary processing to directly provide a velocity measurement and hence the resulting output equation can be written as:
\begin{align*}
y = \matrix{1 & 0 & 0 & 0 \\ 0 & 1 & 0 & 0 \\ 0 & 1 & 0 & 0\\ 0 & 0 & 1 & 0 \\ 0 & 0 & 0 & 1} \matrix{x \\ v \\ \theta \\ \omega }  + \matrix{\psi_1 \\ \psi_2 \\ \psi_3 \\ \psi_4 \\ \psi_5},
\end{align*}
where $\psi_i$ is the measurement noise of the $i$the sensor which is assumed to be gaussian with zero norm and finite covariance.

By applying Proposition \ref{lemma:rip} for different values of $s$, we conclude that the UGV can be resilient only to one attack on any of the two encoders. Attacking any other sensor precludes the system from being $2s$-sparse observable. 

Figure \ref{fig:performance_tank} shows the performance of the proposed algorithms under
different attacks on the UGV motor encoders. The attacker alternates between corrupting the left and the right encoder measurements as shown in Figure \ref{fig:tank_attack_right} and Figure \ref{fig:tank_attack_left}. Three different types of attacks are considered. First, the attacker
corrupts the sensor signal with random noise. The next attack consists of a step function followed by a ramp. Finally a replay-attack is mounted by replaying the previously measured
UGV velocity.

The UGV vehicle goal is to move $5$m along a straight line, stop and perform a $90^o$ rotation and repeat this pattern 3 times until it traces a square and returns to its original position and orientation. In Figures \ref{fig:tank_state_1_etpg} and \ref{fig:tank_state_2_etpg} we show the result of using the ETPG algorithm and the ETPL observer to reconstruct the state under sensor attacks. The reconstructed state is used by a linear feedback tracking controller forcing the UGV to track the desired square trajectory.

The reconstructed position and velocity are shown in Figures
\ref{fig:tank_state_1_etpg} and \ref{fig:tank_state_2_etpg}. These figures show that both algorithms are able to successfully reconstruct the state and hence the UGV is able to reach its goal despite the attacks. Moreover, we observe that the ETPL observer is less sensitive to noise compared to the ETPG. This follows from the fact the ETPL observer ``averages out'' the noise by using all the available sensor data as it becomes available.

Recall that the attack model in Section \ref{sec:problem} requires the set of attacked sensors to remain constant over time. However, Figure \ref{fig:performance_tank} shows the proposed algorithms correctly constructing the state even though this assumption is violated. This is due to the fact that the period during which only one sensor is attacked is sufficiently long compared with the time it takes for the algorithms to converge.

\begin{figure} 
\centering 
\subfigure[Reconstructed position versus ground truth. 
]{\label{fig:tank_state_1_etpg}
\resizebox{0.47\textwidth}{!}{
\includegraphics{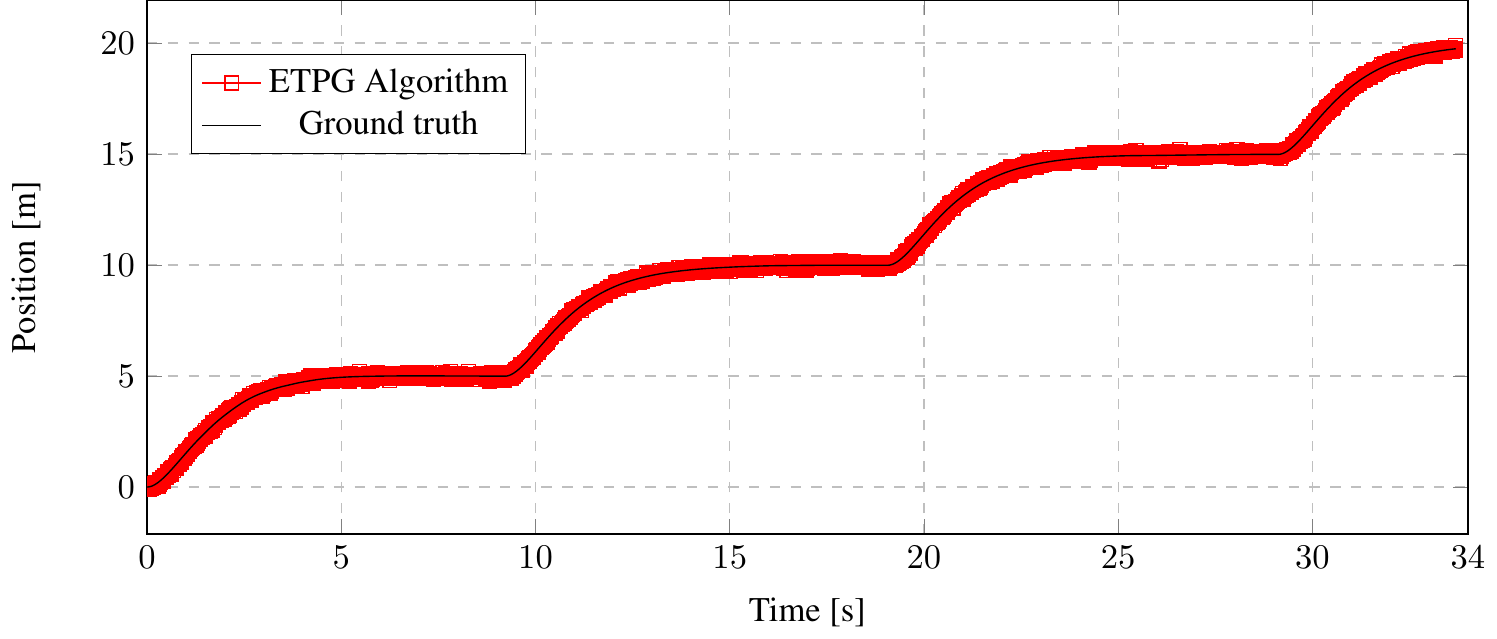}
		}
\resizebox{0.47\textwidth}{!}{
\includegraphics{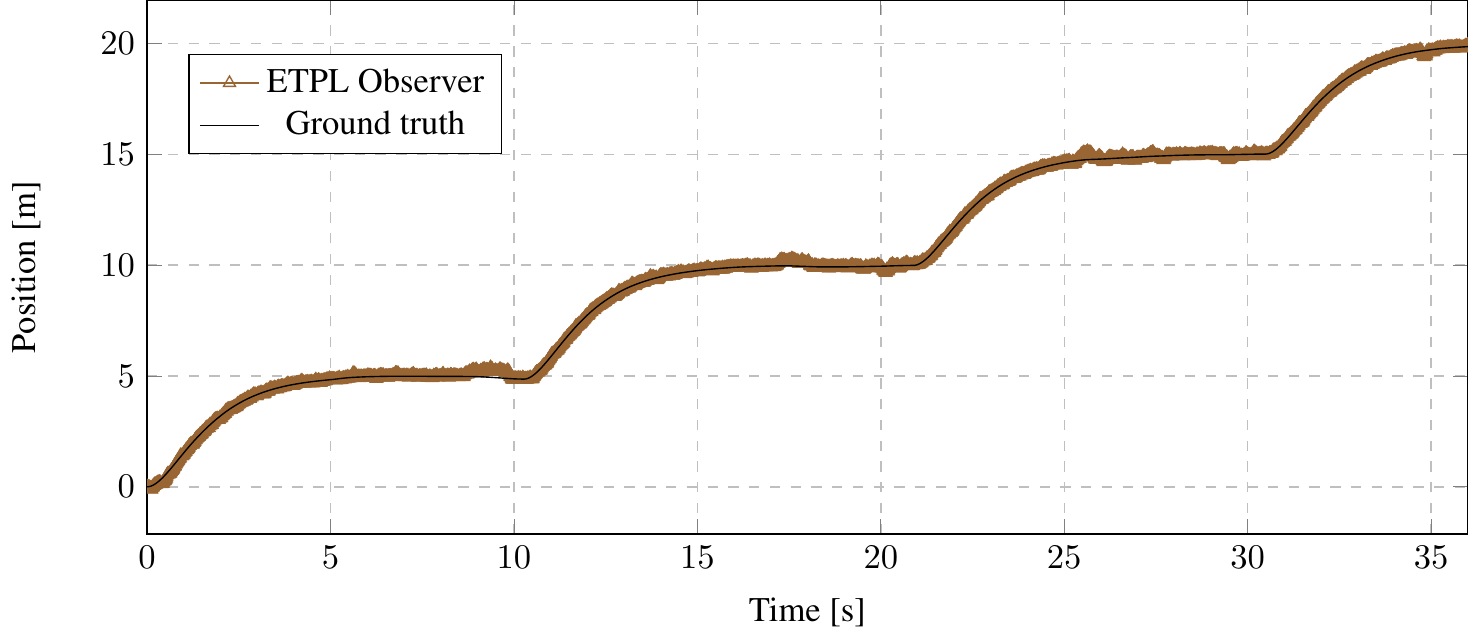}
		}
}\\
\subfigure[Reconstructed velocity versus ground truth. 
]{\label{fig:tank_state_2_etpg}
\resizebox{0.47\textwidth}{!}{
\includegraphics{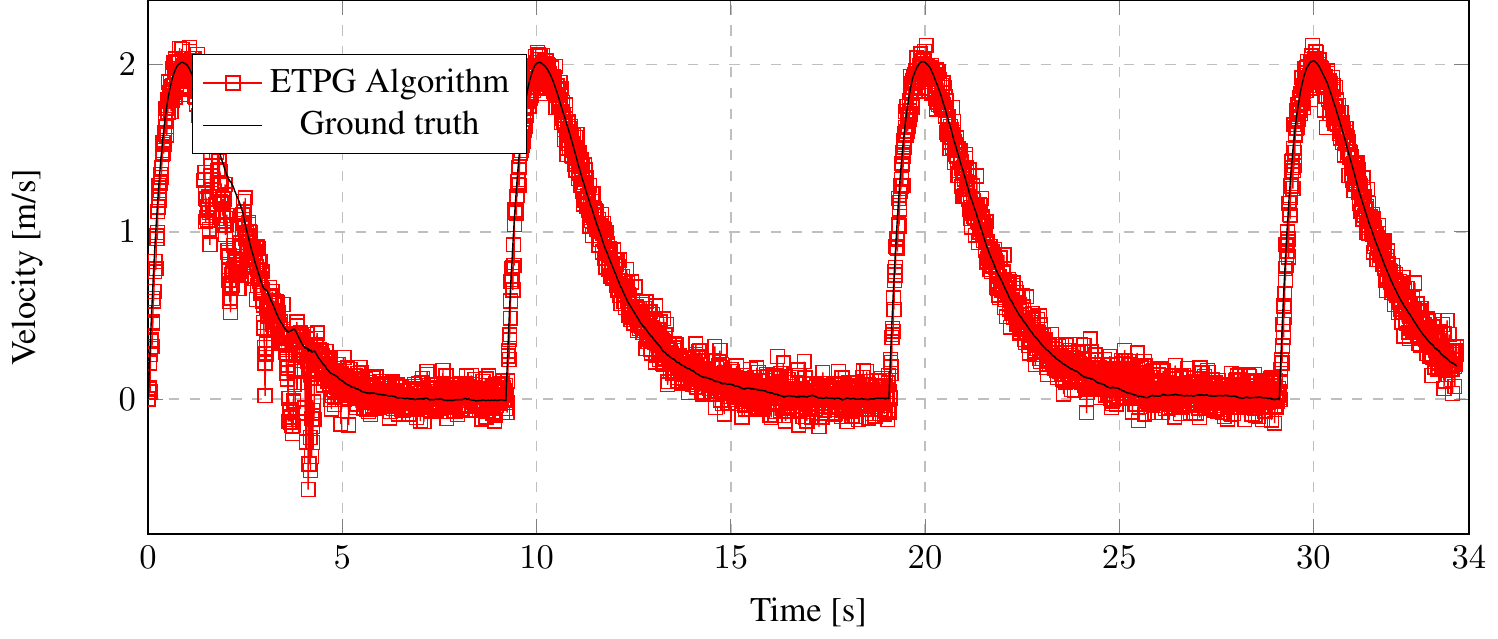}
		}
\resizebox{0.47\textwidth}{!}{
\includegraphics{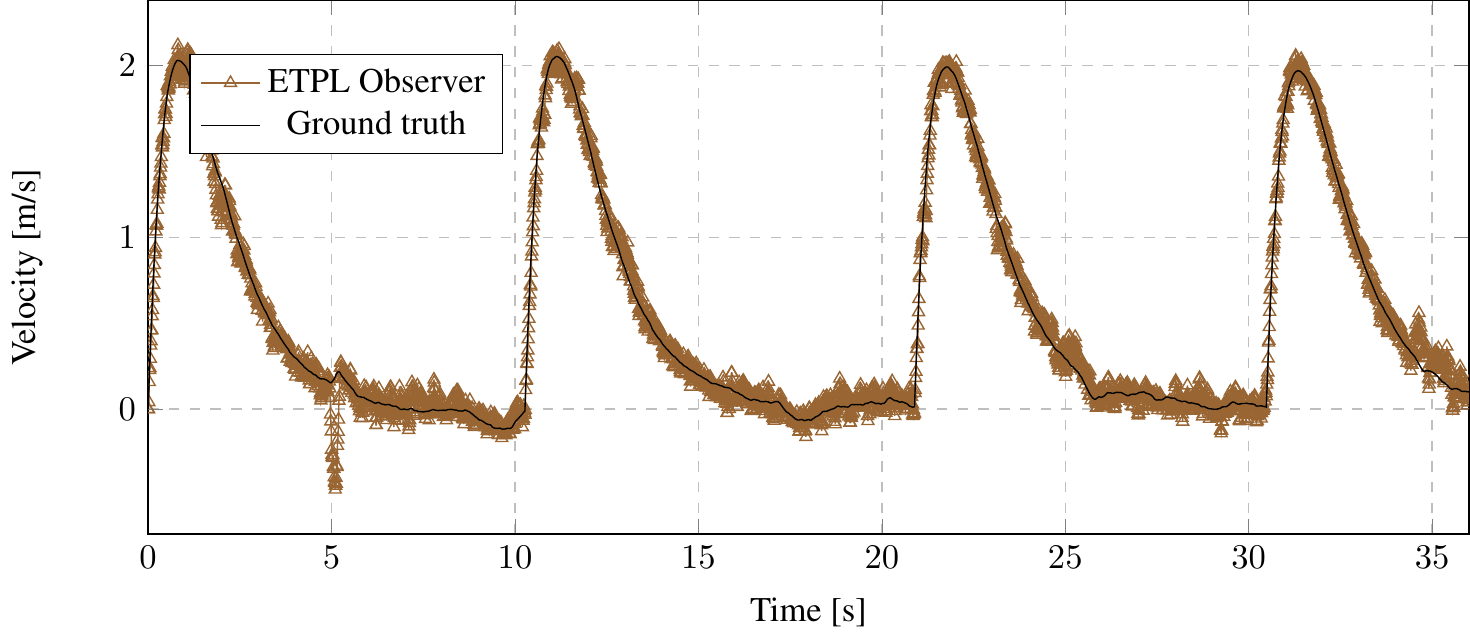}
		}
}\\
\subfigure[Reconstructed attack on left encoder versus ground truth.]
{\label{fig:tank_attack_left}
\resizebox{0.47\textwidth}{!}{
\includegraphics{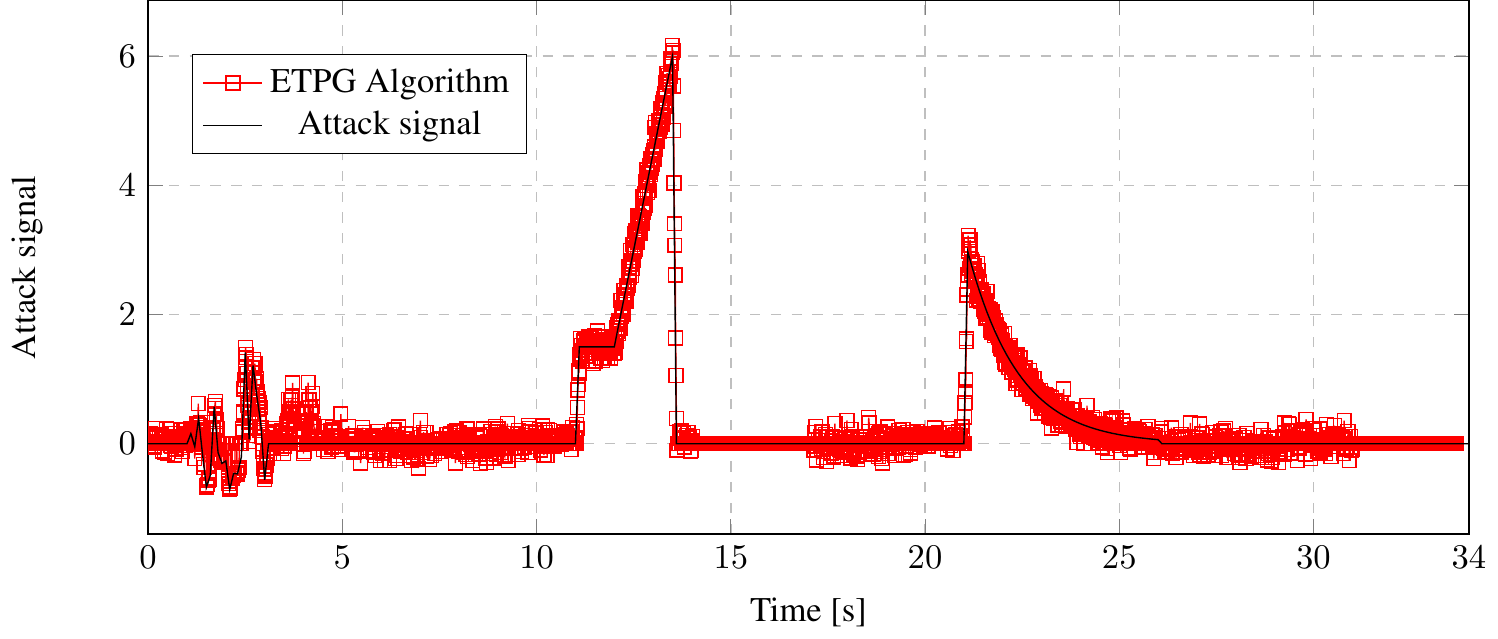}
		}
\resizebox{0.47\textwidth}{!}{
\includegraphics{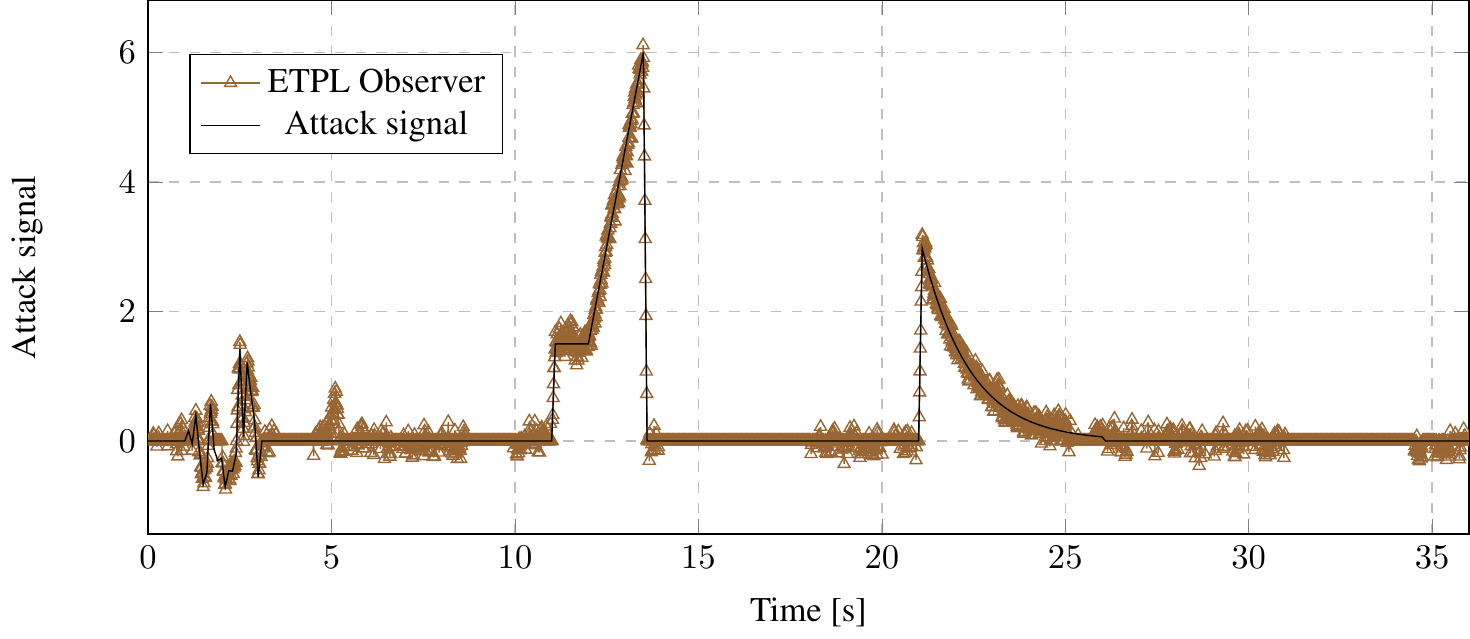}
		}
}\\
\subfigure[Reconstructed attack on right encoder versus ground truth.]
{\label{fig:tank_attack_right}
\resizebox{0.47\textwidth}{!}{
\includegraphics{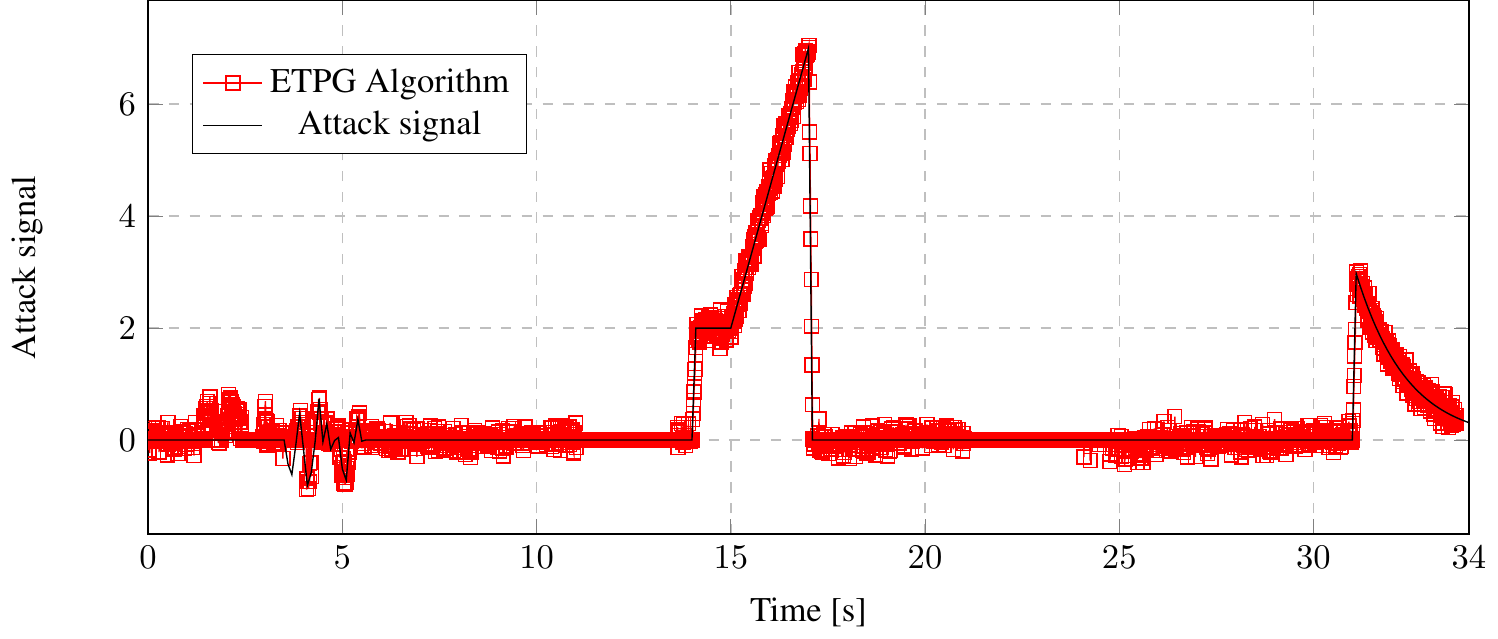}
		}
\resizebox{0.47\textwidth}{!}{
\includegraphics{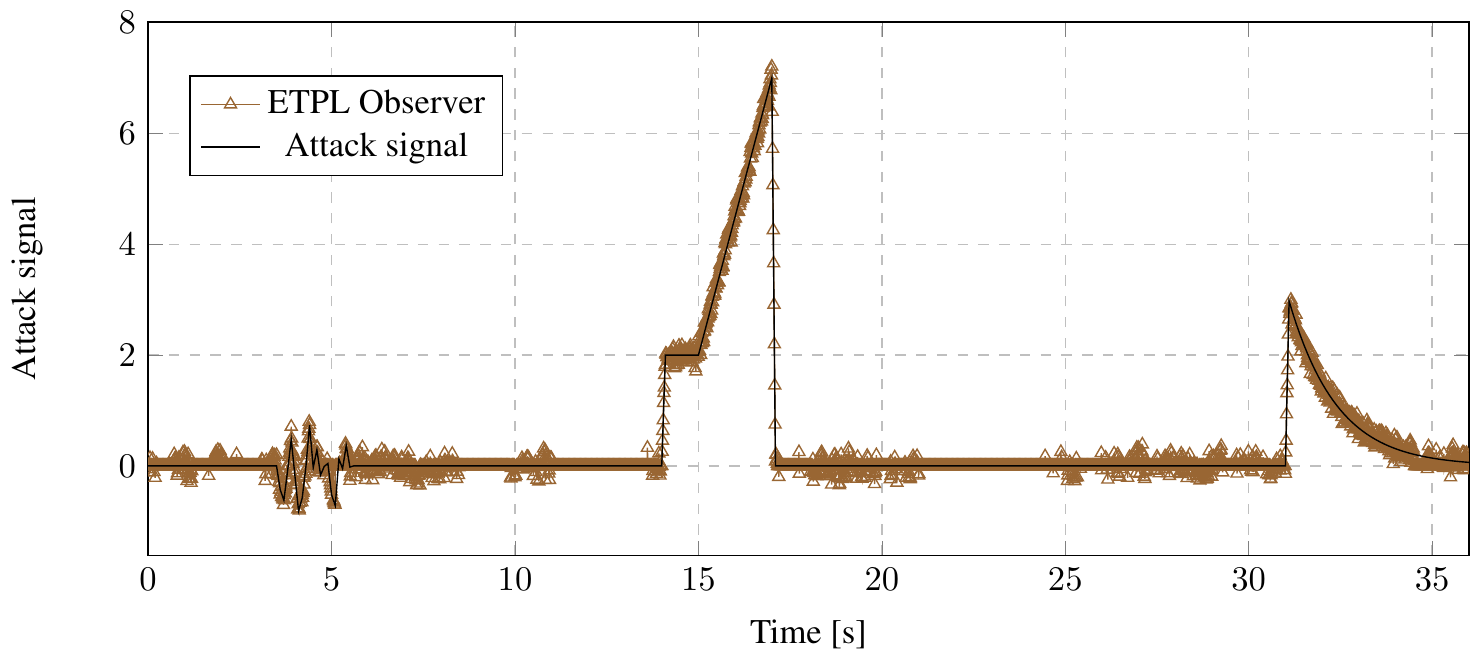}
	}
}
\caption{Performance of the UGV controller in the cases where no attack takes place versus the case where the attack signal is applied to the UGV encoders. The objective is to move $5$ m, stop and perform a $90^o$ rotation and repeat this pattern to follow a square path. The controller uses ETPG algorithm and ETPL Observer to reconstruct the UGV states. In both cases we show the linear position (top), linear velocity (middle), and the reconstruction of the attack signal (bottom).}
\label{fig:performance_tank}
\end{figure}

\subsection{Computational Performance}

We compare the efficiency of the proposed ETPG and ETPL
algorithms against the $L_1/L_r$ decoder introduced in \cite{HamzaAllerton}. To perform this
comparison, we randomly generated $100$ systems with $n = 20$ and $p = 25$ and 
simulated them against an increasing number of
attacked sensors ranging from $0$ to $12$. 
For each test case we generated a random support set for the attack vector, random attack signal and random initial conditions. 
Averaged results for the different numbers of attacked sensors are shown in Figure \ref{fig:time}. Although we claim no statistical significance, the results in Figure~\ref{fig:time} are characteristic of the many simulations performed by the authors.
The $L_1/L_r$ decoder is implemented using CVX while
the ETPG and ETPL algorithms are direct implementations of Algorithms
\ref{alg:etpg} and \ref{alg:etpl} in Matlab. The tests were performed on a desktop equipped with an Intel Core i7 processor operating at 3.4 GHz and 8 GBytes of memory.

Note that the ETPL observer, as required by any solution to Problem \ref{prob:observer}, is an asymptotic observer, i.e., the reconstructed state converges to the true state asymptotically. This should be contrasted with the ETPG algorithm where a single execution of Algorithm \ref{alg:etpg} is sufficient to construct an estimate which is  sufficiently close\footnote{Although the outer loop of Algorithm \ref{alg:etpg} requires $V\left(\hat{z}_\Pi^{(k-1)}\right)\le 0$ for termination, our implementation used instead $V\left(\hat{z}_\Pi^{(k-1)}\right)\le \epsilon$ with $\epsilon = 10^{-6}$.} to the system state. Hence, to compare the ETPG and ETPL algorithms we define the \emph{execution time} as the time needed by Algorithms \ref{alg:etpg} and \ref{alg:etpl} to terminate and the \emph{convergence time} as the time needed by each algorithm from the start of the execution until the estimate becomes $\epsilon$-close to the system state (with $\epsilon$ is set to $\epsilon = 10^{-6}$ in this example). Note that the execution time affects the choice of the sampling period when the algorithm is deployed while the convergence time reflects the performance of each of the algorithms.

In Figure~\ref{fig:execTime} we can appreciate how both the ETPG and the ETPL algorithms outperform the  $L_1/L_r$ decoder by an order of magnitude in execution
time. We also observe that the ETPG algorithm requires more execution time
compared to the ETPL observer. This follows from the existence of the outer-loop
in the ETPG algorithm which requires the Lyapunov function to reach zero before
termination. In  Figure \ref{fig:convgTime} we show the convergence time for each of the three algorithms. It follows from the nature of the $L_1/L_r$ decoder and the ETPG algorithm that the execution time and the convergence time are both equal. This is not the case for the ETPL observer which requires a longer convergence time compared to the ETPG algorithm. These two figures illustrate the tradeoff between execution timing and performance.
\begin{figure}
	\centering
	\subfigure[Execution time.]
	{\label{fig:execTime}
	\resizebox{0.47\textwidth}{!}{
		\begin{tikzpicture}
			\begin{axis}[
				width = 9cm,
				height = 5cm,
			    	xlabel=Number of attacked sensors $s$,		
			    	ylabel=Execution time {[s]},
			    	xmin = 1,
			    	xmax = 13,
			    	legend style={at={(0.7,0.7)},anchor=north,legend columns=1},
			    	ymajorgrids=true,
				xmajorgrids=true,
			    	grid style=dashed,	]
				\addplot table{time_cvx.txt};
				\addplot table{time_etpg.txt};
				\addplot[color={brown!70!black},mark=triangle] table{time_etpl.txt};
				\legend{$L_r/L_2$ Decoder,ETPG Algorithm, ETPL Observer}
			\end{axis}
		\end{tikzpicture}
		}
	}
	\subfigure[Convergence time.] 
	{\label{fig:convgTime}
\resizebox{0.47\textwidth}{!}{
		\begin{tikzpicture}
			\begin{axis}[
				width = 9cm,
				height = 5cm,
			    	xlabel=Number of attacked sensors $s$,		
			    	ylabel=Convergence time {[s]},
			    	xmin = 1,
			    	xmax = 13,
			    	legend style={at={(0.7,0.7)},anchor=north,legend columns=1},
			    	ymajorgrids=true,
				xmajorgrids=true,
			    	grid style=dashed,	]
				\addplot table{timeConvergence_cvx.txt};				
				\addplot table{timeConvergence_etpg.txt};
				\addplot[color={brown!70!black},mark=triangle] table{timeConvergence_etpl.txt};
				\legend{$L_r/L_2$ Decoder, ETPG Algorithm, ETPL Observer}
			\end{axis}
		\end{tikzpicture}
		}
	}
	 \caption{Timing analysis of the $L_1/L_2$ decoder versus the ETPG and ETPL
	 algorithms for different numbers of attacked sensors. Subfigure (a) shows the time required to execute each of the three algorithms.
	 Subfigure (b) shows the time required for the estimate, computed by each of the algorithms, to become $\epsilon$-close to the system state for $\epsilon=10^{-6}$.}
	\label{fig:time}	
\end{figure}
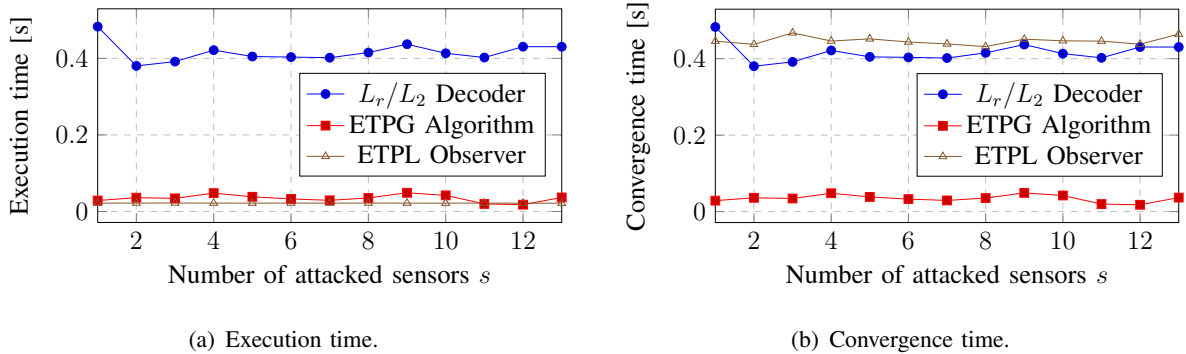

\section{Conclusions}
\label{sec:conclusion}
In this paper we considered the problem of designing computationally
efficient algorithms for state reconstruction under sparse adversarial attacks/noise.  We characterized the solvability of this problem by using the notion of sparse observability and proposed two algorithms for state reconstruction. To improve the timing performance of the proposed algorithms, we adopted an event-triggered approach that determines on-line how many gradient steps should be executed per projection on the set of constraints. These algorithms can be further improved along multiple directions such as dynamically adjusting the step size or using more refined gradient algorithms, such as conjugated gradient.

\bibliographystyle{IEEEtran}
\bibliography{bibliography2}

\end{document}